\documentclass[12pt]{article} 
\usepackage{amsmath, amssymb}
\usepackage{amsthm} 
\theoremstyle{plain} 
\newtheorem{theorem}{\indent\sc Theorem}[section]

\newtheorem{corollary}[theorem]{\indent\sc Corollary}
\newtheorem{proposition}[theorem]{\indent\sc Proposition}

\theoremstyle{definition} 

\newtheorem{remark}[theorem]{\indent\sc Remark}
\newtheorem{example}[theorem]{\indent\sc Example}

\newcommand\on{\operatorname}

\newcommand\R{\on{Ric}}

\title{On some $3$-dimensional almost $\eta$-Ricci solitons \\ with diagonal metrics}
\author{Adara M. Blaga}

\date{}
\begin{document}

\maketitle

\markboth{{\small\it {\hspace{1cm} On some $3$-dimensional almost $\eta$-Ricci solitons with diagonal metrics}}}{\small\it{On some $3$-dimensional almost $\eta$-Ricci solitons with diagonal metrics
\hspace{1cm}}}


\begin{abstract}
We study some properties of a $3$-dimensional manifold with a diagonal Riemannian metric as an almost $\eta$-Ricci soliton from the following points of view: under certain assumptions, we determine the potential vector field if $\eta$ is given; we get constraints on the metric when the potential vector field has a particular expression; we compute the defining functions of the soliton when both the potential vector field and the $1$-form are prescribed. Moreover, we find conditions for the manifold to be flat. Based on the theoretical results, we provide examples.
\end{abstract}

\section{Introduction}

The solitons' theory has been lately intensively investigated. The stationary solutions of the Ricci flow, namely, the \textit{Ricci solitons}, still represent a very actual topic to be studied from various points of view. Problems like linear stability of compact Ricci solitons, curvature estimates, rigidity results for gradient Ricci solitons, etc. have been recently treated by Huai-Dong Cao \textit{et. al.} in \cite{cao1, cao2,cao3}. As a generalization of the notion of Ricci soliton given by Hamilton in \cite{ham}, an \textit{almost $\eta$-Ricci soliton} \cite{b, adara, kimura} is a Rieman\-nian manifold $(M,\tilde g)$ with a smooth vector field $V$ which satisfies the following equation
\begin{equation}\label{11}
\frac{1}{2}\pounds_{V}\tilde g+\R+\lambda \tilde g+\mu \eta \otimes \eta =0,
\end{equation}
where $\lambda$ and $\mu$ are two smooth functions on $M$ with $\mu\neq 0$, $\R$ is the Ricci curvature tensor field, $\pounds _{V}\tilde g$ is the Lie derivative of the metric $\tilde g$ in the direction of $V$, and $\eta$ is a $1$-form on $M$. If $V$ is a Killing vector field, i.e., $\pounds _{V}\tilde g=0$, then the soliton is called \textit{trivial}, and in this case, $(M, \tilde g)$ is just a quasi-Einstein manifold.
On the other hand, if $\eta=0$, then $(M, \tilde g,V,\lambda)$ is called an \textit{almost Ricci soliton} \cite{pi}. In this case, the soliton is said to be \textit{shrinking}, \textit{steady} or \textit{expanding} according as $\lambda$ is negative, zero or positive, respectively \cite{chlu}. In the above cases, if $\lambda$ and $\mu$ are real numbers, then we drop "almost".

The aim of the present paper is to describe a $3$-dimensional manifold endowed with a diagonal Riemannian metric as an almost $\eta$-Ricci soliton. More precisely, under certain assumptions, we determine the potential vector field $V$ when $\eta$ is given, we find the conditions that must be satisfied by the Riemannian metric when the potential vector field has a particular expression, and we compute the defining functions $\lambda$ and $\mu$ when both the potential vector field and the $1$-form are prescribed. Based on the theoretical results, we construct examples, among which the $3$-dimensional $Sol_3$ and $\mathbb H^2\times \mathbb R$ Lie groups. An analogous study for the canonical metric has been previously done by the present author in \cite{bl}. More recently, some particular Riemannian manifolds, namely, the $3$-dimensional $Sol_3$ and $\mathbb H^2\times \mathbb R$ Lie groups, have been described as Ricci solitons by Belarbi, Atashpeykara and Haji-Badali in \cite{bel1,bel2,ata} from a similar point of view.

\section{The flatness condition}

We consider $I=I_1\times I_2\times I_3\subseteq \mathbb R^3$, where $I_i\subseteq \mathbb R$, $i\in \{1,2,3\}$, are open intervals, endowed with a diagonal Riemannian metric $\tilde{g}$ given by
$$\tilde{g}=\frac{1}{f_1^2}dx^1\otimes dx^1+\frac{1}{f_2^2}dx^2\otimes dx^2+dx^3\otimes dx^3,$$
where $f_1$ and $f_2$ are two smooth functions nowhere zero on $I$, and $x^1,x^2,x^3$ are the standard coordinates in $\mathbb R^3$. Let
$$\left\{E_1:=f_1\frac{\partial}{\partial x^1}, \ \ E_2:=f_2\frac{\partial}{\partial x^2}, \ \ E_3:=\frac{\partial}{\partial x^3}\right\}$$
be a local orthonormal frame, and, for the sake of simplicity, we will make the following notations:
$$\frac{f_2}{f_1}\cdot\frac{\partial f_1}{\partial x^2}=:a, \ \ \frac{1}{f_1}\cdot\frac{\partial f_1}{\partial x^3}=:b, \ \ \frac{f_1}{f_2}\cdot\frac{\partial f_2}{\partial x^1}=:c, \ \ \frac{1}{f_2}\cdot\frac{\partial f_2}{\partial x^3}=:d.$$
Computing the Lie brackets $[X,Y]:=X\circ Y-Y\circ X$, we get
$$[E_1,E_2]=-aE_1+cE_2,\ \ [E_1,E_3]=-bE_1, \ \ [E_2,E_3]=-dE_2.$$
The Levi-Civita connection $\nabla$ of $\tilde g$, deduced from
the Koszul's formula
$$2\tilde g(\nabla_XY,Z)=X(\tilde g(Y,Z))+Y(\tilde g(Z,X))-Z(\tilde g(X,Y))-$$$$-\tilde g(X,[Y,Z])+\tilde g(Y,[Z,X])+\tilde g(Z,[X,Y]),$$
is given by
$$\nabla_{E_1}E_1=aE_2+bE_3, \ \ \nabla_{E_1}E_2=-aE_1, \ \ \nabla_{E_1}E_3=-bE_1, \ \ \nabla_{E_2}E_1=-cE_2,$$$$\nabla_{E_2}E_2=cE_1+dE_3, \ \ \nabla_{E_2}E_3=-dE_2, \ \ \nabla_{E_3}E_1=0, \ \ \nabla_{E_3}E_2=0, \ \ \nabla_{E_3}E_3=0,$$
and the Riemann and Ricci curvature tensor fields
$$R(X,Y)Z:=\nabla_X\nabla_YZ-\nabla_Y\nabla_XZ-\nabla_{[X,Y]}Z, \ \
\R(Y,Z):=\sum_{k=1}^3\tilde g(R(E_k,Y)Z,E_k)
$$ are the following
\begin{align*}
R(E_1,E_2)E_2&=[E_1(c)+E_2(a)-a^2-c^2-bd]E_1+[E_3(c)-cd]E_3,\\
R(E_2,E_1)E_1&=[E_1(c)+E_2(a)-a^2-c^2-bd]E_2+[E_3(a)-ab]E_3,\\
R(E_1,E_3)E_3&=[E_3(b)-b^2]E_1,\\
R(E_2,E_3)E_3&=[E_3(d)-d^2]E_2,\\
R(E_3,E_1)E_1&=[E_3(a)-ab]E_2+[E_3(b)-b^2]E_3,\\
R(E_3,E_2)E_2&=[E_3(c)-cd]E_1+[E_3(d)-d^2]E_3,\\
R(E_1,E_2)E_3&=
\tilde g(R(E_2,E_1)E_1,E_3)E_1-\tilde g(R(E_1,E_2)E_2,E_3)E_2\\
&=[E_3(a)-ab]E_1-[E_3(c)-cd]E_2,\\
R(E_2,E_3)E_1&=\tilde g(R(E_3,E_2)E_2,E_1)E_2-\tilde g(R(E_2,E_3)E_3,E_1)E_3\\
&=[E_3(c)-cd]E_2,\\
R(E_3,E_1)E_2&=
-\tilde g(R(E_3,E_1)E_1,E_2)E_1+\tilde g(R(E_1,E_3)E_3,E_2)E_3\\
&=-[E_3(a)-ab]E_1,\\
\R(E_1,E_1)&=E_1(c)+E_2(a)+E_3(b)-a^2-b^2-c^2-bd,\\
\R(E_2,E_2)&=E_1(c)+E_2(a)+E_3(d)-a^2-c^2-d^2-bd,\\
\R(E_3,E_3)&=E_3(b)+E_3(d)-b^2-d^2,\\
\R(E_1,E_2)&=0,\ \ \R(E_1,E_3)=E_3(c)-cd, \ \ \R(E_2,E_3)=E_3(a)-ab.
\end{align*}

We shall determine the conditions that the two functions $f_1$ and $f_2$ must satisfy for the manifold to be flat, exploring the cases when each of the two functions depends on a single variable. In this paper, if a function $h$ on $\mathbb R^3$ depends only on some of its variables, then we will write in its argument only that variables in order to emphasize this fact, for example, $h(x^i)$, $h(x^i,x^j)$.

\begin{proposition}\label{ps}
If $f_i=f_i(x^i)$ for $i\in \{1,2\}$, then $(I,\tilde g)$ is a flat Riemannian manifold.
\end{proposition}
\begin{proof}
In this case, $a=b=c=d=0$, and $R(E_i,E_j)E_k=0$ for any $i,j,k\in \nolinebreak\{1,2,3\}$.
\end{proof}

\begin{proposition}\label{ps1}
If $f_i=f_i(x^3)$ for $i\in \{1,2\}$, then the following assertions are equivalent:

(1) $(I,\tilde g)$ is a flat Riemannian manifold;

(2) $f_1$ and $f_2$ are constant, or $f_i=k_i\in \mathbb R\setminus\{0\}$ and
$f_j(x^3)=\frac{\displaystyle c_1}{\displaystyle x^3-c_2}$, where 
\linebreak 
$c_1\in \mathbb R\setminus \{0\}$, $c_2\in \mathbb R\setminus I_3$, and $i\neq j$.
\end{proposition}
\begin{proof}
In this case, $a=0$, $b=\frac{\displaystyle f_1'}{\displaystyle f_1}$, $c=0$, $d=\frac{\displaystyle f_2'}{\displaystyle f_2}$, and
$$R(E_1,E_2)E_2=-bdE_1, \ \ R(E_2,E_1)E_1=-bdE_2,$$
$$R(E_1,E_3)E_3=(b'-b^2)E_1, \ \ R(E_3,E_1)E_1=(b'-b^2)E_3,$$
$$R(E_2,E_3)E_3=(d'-d^2)E_2, \ \ R(E_3,E_2)E_2=(d'-d^2)E_3.$$
Then, $R=0$ if and only if
\begin{equation*}
\left\{
    \begin{aligned}
  &     bd=0 \\
  &     b'=b^2 \\
  &     d'=d^2
    \end{aligned}
  \right.,
\end{equation*}
that is,
\begin{equation*}
\left\{
    \begin{aligned}
 &      f_1'f_2'=0 \\
 &      \left(\frac{\displaystyle f_1'}{\displaystyle f_1}\right)'=\left(\frac{\displaystyle f_1'}{\displaystyle f_1}\right)^2 \\
 &      \left(\frac{\displaystyle f_2'}{\displaystyle f_2}\right)'=\left(\frac{\displaystyle f_2'}{\displaystyle f_2}\right)^2
    \end{aligned}
  \right..
\end{equation*}
\pagebreak

\noindent
Therefore,

$\bullet$ $f_1'=0$ and $f_2'=0$, or

$\bullet$ $f_1'=0$ and $-\frac{\displaystyle 1}{\frac{\displaystyle f_2'(x^3)}{\displaystyle f_2(x^3)}}=x^3+c_0$ \Big(i.e., $f_2(x^3)=\frac{\displaystyle c_1}{\displaystyle x^3-c_2}$, where $c_1\in \mathbb R\setminus \{0\}$, $c_2\in \mathbb R\setminus I_3$\Big), or, similarly

$\bullet$ $f_1(x^3)=\frac{\displaystyle c_1}{\displaystyle x^3-c_2}$, where $c_1\in \mathbb R\setminus \{0\}$, $c_2\in \mathbb R\setminus I_3$ and $f_2'=0$.
\end{proof}

\begin{corollary}
For $I=\mathbb R^3$, if $f_i=f_i(x^3)$ for $i\in \{1,2\}$, then the following assertions are equivalent:

(1) $(\mathbb R^3,\tilde g)$ is a flat Riemannian manifold;

(2) $f_1$ and $f_2$ are constant.
\end{corollary}
\begin{proof}
It follows immediately from Proposition \ref{ps1}.
\end{proof}

\begin{proposition}\label{ps11}
If $f_1=f_1(x^1)$, $f_2=f_2(x^3)$, then the following assertions are equivalent:

(1) $(I,\tilde g)$ is a flat Riemannian manifold;

(2) $f_2$ is constant, or $f_2(x^3)=\frac{\displaystyle c_1}{\displaystyle x^3-c_2}$, where $c_1\in \mathbb R\setminus \{0\}$, $c_2\in \mathbb R\setminus I_3$.
\end{proposition}
\begin{proof}
In this case, $a=b=c=0$, $d=\frac{\displaystyle f_2'}{\displaystyle f_2}$, and
$$R(E_2,E_3)E_3=(d'-d^2)E_2, \ \ R(E_3,E_2)E_2=(d'-d^2)E_3,$$$$R(E_1,E_2)E_2=R(E_1,E_3)E_3=R(E_2,E_1)E_1=R(E_3,E_1)E_1=0.$$
Then, $R=0$ if and only if
$$d'=d^2,$$
that is,
$$\left(\frac{\displaystyle f_2'}{\displaystyle f_2}\right)'=\left(\frac{\displaystyle f_2'}{\displaystyle f_2}\right)^2.$$
Therefore,

$\bullet$ $f_2'=0$, or

$\bullet$ $-\frac{\displaystyle 1}{\frac{\displaystyle f_2'(x^3)}{\displaystyle f_2(x^3)}}=x^3+c_0$ \Big(i.e., $f_2(x^3)=\frac{\displaystyle c_1}{\displaystyle x^3-c_2}$, where $c_1\in \mathbb R\setminus \{0\}$, $c_2\in \mathbb R\setminus I_3$\Big).
\end{proof}

\begin{corollary}
For $I=\mathbb R^3$, if $f_1=f_1(x^1)$, $f_2=f_2(x^3)$, then the following assertions are equivalent:

(1) $(\mathbb R^3,\tilde g)$ is a flat Riemannian manifold;

(2) $f_2$ is constant.
\end{corollary}
\begin{proof}
It follows immediately from Proposition \ref{ps11}.
\end{proof}

\begin{proposition}\label{ps11ax}
If $f_i=f_i(x^2)$ for $i\in \{1,2\}$, then the following assertions are equivalent:

(1) $(I,\tilde g)$ is a flat Riemannian manifold;

(2) $f_1$ and $f_2$ satisfy the equation
$$\frac{\displaystyle f_1'}{\displaystyle f_1} \frac{\displaystyle f_2'}{\displaystyle f_2}+\frac{\displaystyle f_1''}{\displaystyle f_1}=2\left(\frac{\displaystyle f_1'}{\displaystyle f_1}\right)^2;$$

(3) $f_1$ is constant, or $f_1'$ is nowhere zero and $f_2=c_0\frac{\displaystyle f_1^2}{\displaystyle f_1'}$, where $c_0\in \mathbb R\setminus \{0\}$.
\end{proposition}
\begin{proof}
In this case, $a=f_2\frac{\displaystyle f_1'}{\displaystyle f_1}$, $b=c=d=0$, and
$$R(E_1,E_2)E_2=[E_2(a)-a^2]E_1, \ \ R(E_2,E_1)E_1=[E_2(a)-a^2]E_2,$$
$$R(E_1,E_3)E_3=R(E_3,E_1)E_1=R(E_2,E_3)E_3=R(E_3,E_2)E_2=0.$$
Then, $R=0$ if and only if
$$f_2a'=a^2,$$
that is,
$$f_2f_2'\frac{\displaystyle f_1'}{\displaystyle f_1}+f_2^2\frac{\displaystyle f_1''f_1-(f_1')^2}{\displaystyle f_1^2}=f_2^2\left(\frac{\displaystyle f_1'}{\displaystyle f_1}\right)^2,$$
which is equivalent to
$$\frac{\displaystyle f_1'}{\displaystyle f_1}\frac{\displaystyle f_2'}{\displaystyle f_2}+\frac{\displaystyle f_1''}{\displaystyle f_1}=2\left(\frac{\displaystyle f_1'}{\displaystyle f_1}\right)^2.$$
If $f_1'= 0$, then $f_1$ is a constant function. If $f_1'\neq 0$ (hence, if it is nowhere zero), then the previous relation can be written as
$$\frac{\displaystyle f_2'}{\displaystyle f_2}=\frac{2\left(\frac{\displaystyle f_1'}{\displaystyle f_1}\right)^2-\frac{\displaystyle f_1''}{\displaystyle f_1}}{\frac{\displaystyle f_1'}{\displaystyle f_1}}.$$
Let us notice that
$$\frac{\displaystyle f_2'}{\displaystyle f_2}=\frac{\left(\frac{\displaystyle f_1'}{\displaystyle f_1}\right)^2-\left(\frac{\displaystyle f_1'}{\displaystyle f_1}\right)'}{\frac{\displaystyle f_1'}{\displaystyle f_1}}=\frac{\displaystyle f_1'}{\displaystyle f_1}-\frac{\left(\frac{\displaystyle f_1'}{\displaystyle f_1}\right)'}{\frac{\displaystyle f_1'}{\displaystyle f_1}},$$
which, by integration, gives
$$\ln|f_2|=\ln|f_1|-\ln \left|\frac{f_1'}{f_1}\right|+k=\ln\left(e^{k}\frac{f_1^2}{|f_1'|}\right),$$
where $k\in \mathbb R$;
therefore,
$f_2=c_0\frac{\displaystyle f_1^2}{\displaystyle f_1'}$, where $c_0\in \mathbb R\setminus \{0\}$.
\end{proof}

\begin{corollary}
For $I=\mathbb R^3$, if $f_1=f_2=:f(x^2)$, then the following assertions are equivalent:

(1) $(\mathbb R^3,\tilde g)$ is a flat Riemannian manifold;

(2) $f(x^2)=c_1e^{c_2x^2}$, where $c_1\in \mathbb R\setminus\{0\}$, $c_2\in\mathbb R$.
\end{corollary}
\begin{proof}
(1) is equivalent to $ff''-(f')^2=0$, that is, $\left(\frac{\displaystyle f'}{\displaystyle f}\right)'=0$, with the solution
$f(x^2)=c_1e^{c_2x^2}$, where $c_1\in \mathbb R\setminus\{0\}$, $c_2\in\mathbb R$,
hence we get the conclusion.
\end{proof}

\begin{proposition}\label{ps11b}
If $f_1=f_1(x^2)$, $f_2=f_2(x^1)$, then the following assertions are equivalent:

(1) $(I,\tilde g)$ is a flat Riemannian manifold;

(2) $f_1$ and $f_2$ satisfy the equation
$$\frac{\displaystyle f_1''f_1-2(f_1')^2}{\displaystyle f_1^4}=-\frac{\displaystyle  f_2''f_2-2(f_2')^2}{\displaystyle f_2^4} \ =\textit{constant}.$$
\end{proposition}
\begin{proof}
In this case, $a=f_2\frac{\displaystyle f_1'}{\displaystyle f_1}$, $b=0$, $c=f_1\frac{\displaystyle f_2'}{\displaystyle f_2}$, $d=0$, and
$$R(E_1,E_2)E_2=[E_1(c)+E_2(a)-a^2-c^2]E_1, \ \ R(E_2,E_1)E_1=[E_1(c)+E_2(a)-a^2-c^2]E_2,$$
$$R(E_1,E_3)E_3=R(E_3,E_1)E_1=R(E_2,E_3)E_3=R(E_3,E_2)E_2=0.$$
Then, $R=0$ if and only if
$$ f_1\frac{\partial c}{\partial x^1}+f_2\frac{\partial a}{\partial x^2}=a^2+c^2,$$
that is,
$$f_1^2\frac{\displaystyle f_2''f_2-(f_2')^2}{\displaystyle f_2^2}+f_2^2\frac{\displaystyle f_1''f_1-(f_1')^2}{\displaystyle f_1^2}=f_2^2\left(\frac{\displaystyle f_1'}{\displaystyle f_1}\right)^2+f_1^2\left(\frac{\displaystyle f_2'}{\displaystyle f_2}\right)^2,$$
which is equivalent to
$$\frac{\displaystyle f_1''f_1-2(f_1')^2}{\displaystyle f_1^4}=-\frac{\displaystyle  f_2''f_2-2(f_2')^2}{\displaystyle f_2^4}.$$
Since $f_1$ depends only on $x^2$ and $f_2$ depends only on $x^1$, we deduce that the above ratio must be a constant.
\end{proof}

\begin{remark}
We shall now look on the condition (2) from Proposition \ref{ps11b} satisfied by the two functions, namely
$$\frac{\displaystyle f''f-2(f')^2}{\displaystyle f^4}=k \in \mathbb R,$$
when $f$ is a real function defined on a real interval.
Let us notice that
$$\frac{\displaystyle f''f-2(f')^2}{\displaystyle f^4}=-\left(\frac{1}{f}\right)''\frac{1}{f}.$$
Denoting by $h:=\frac{\displaystyle 1}{\displaystyle f}$, we have
$$h''h=-k.$$
Let $r\in \mathbb R$ and let $J\subseteq \mathbb R$ be an open interval such that $0\notin J$ and $-2k\ln |y|+r>0$ for any $y\in J$. Let $F$ be an antiderivative on $J$ of the function
$$y\mapsto \frac{1}{\sqrt{-2k\ln |y|+r}}.$$
Then, $F'(y)>0$ for any $y\in J$; therefore, $F$ is strictly increasing on $J$, hence, it is invertible onto its image. Let $\varepsilon\in\{\pm 1\}$ and let $I_J:=\varepsilon\left(F(J)-c_0\right)$, where
$c_0\in \mathbb R$. Then, $$h:I_J\rightarrow\mathbb R, \ \ h(x):=F^{-1}(\varepsilon x+c_0)$$
satisfies $h''(x)h(x)=-k$ for any $x\in I_J$, and we get
$$f:I_J\rightarrow\mathbb R, \ \ f(x)=\frac{1}{F^{-1}(\varepsilon x+c_0)}.$$
\end{remark}

\begin{proposition}\label{ps11a}
If $f_1=f_1(x^2)$, $f_2=f_2(x^3)$, then the following assertions are equivalent:

(1) $(I,\tilde g)$ is a flat Riemannian manifold;

(2) $f_1$ and $f_2$ are constant, or $f_1=k_1\in \mathbb R\setminus \{0\}$ and $f_2(x^3)=\frac{\displaystyle c_1}{\displaystyle x^3-c_2}$, where $c_1\in \mathbb R\setminus \{0\}$, $c_2\in \mathbb R\setminus I_3$, or $f_2=k_2\in \mathbb R\setminus \{0\}$ and $f_1(x^2)=\frac{\displaystyle c_1}{\displaystyle x^2-c_2}$, where $c_1\in \mathbb R\setminus \{0\}$, $c_2\in \mathbb R\setminus I_2$.
\end{proposition}
\begin{proof}
In this case, $a=f_2\frac{\displaystyle f_1'}{\displaystyle f_1}$, $b=c=0$, $d=\frac{\displaystyle f_2'}{\displaystyle f_2}$, and
$$R(E_1,E_2)E_2=[E_2(a)-a^2]E_1, \ \ R(E_2,E_1)E_1=[E_2(a)-a^2]E_2+E_3(a)E_3,$$
$$R(E_2,E_3)E_3=[E_3(d)-d^2]E_2,\ \
R(E_3,E_2)E_2=[E_3(d)-d^2]E_3,$$
$$R(E_3,E_1)E_1=E_3(a)E_2,\ \
R(E_1,E_2)E_3=E_3(a)E_1,\ \
R(E_3,E_1)E_2=-E_3(a)E_1,$$
$$R(E_1,E_3)E_3=R(E_2,E_3)E_1=0.$$
Then, $R=0$ if and only if
\begin{equation*}
\left\{
    \begin{aligned}
 &      f_1'f_2'=0 \\
 &      \left(\frac{\displaystyle f_1'}{\displaystyle f_1}\right)'=\left(\frac{\displaystyle f_1'}{\displaystyle f_1}\right)^2 \\
 &      \left(\frac{\displaystyle f_2'}{\displaystyle f_2}\right)'=\left(\frac{\displaystyle f_2'}{\displaystyle f_2}\right)^2
    \end{aligned}
  \right..
\end{equation*}
Therefore,

$\bullet$ $f_1'=0$ and $f_2'=0$, or

$\bullet$ $f_1'=0$ and $-\frac{\displaystyle 1}{\frac{\displaystyle f_2'(x^3)}{\displaystyle f_2(x^3)}}=x^3+c_0$ \Big(i.e., $f_2(x^3)=\frac{\displaystyle c_1}{\displaystyle x^3-c_2}$, where $c_1\in \mathbb R\setminus \{0\}$, $c_2\in \mathbb R\setminus I_3$\Big), or, similarly

$\bullet$ $f_1(x^2)=\frac{\displaystyle c_1}{\displaystyle x^2-c_2}$, where $c_1\in \mathbb R\setminus \{0\}$, $c_2\in \mathbb R\setminus I_2$ and $f_2'=0$.
\end{proof}

\begin{corollary}
For $I=\mathbb R^3$, if $f_1=f_1(x^2)$, $f_2=f_2(x^3)$, then the following assertions are equivalent:

(1) $(\mathbb R^3,\tilde g)$ is a flat Riemannian manifold;

(2) $f_1$ and $f_2$ are constant.
\end{corollary}
\begin{proof}
It follows immediately from Proposition \ref{ps11a}.
\end{proof}

\section{$I$ as an almost $\eta$-Ricci soliton}

Let $V=\sum_{k=1}^3V^kE_k$ and $\eta=\sum_{k=1}^3\eta^ke_k$, where $e_k$ is the dual $1$-form of $E_k$ for $k\in \nolinebreak \{1,2,3\}$. From (\ref{11}) we get the equations that define an almost $\eta$-Ricci soliton $(I, \tilde g,V,\lambda,\mu)$:
$$\frac{1}{2}\Big\{E_i(V^j)+E_j(V^i)+\sum_{k=1}^3V^k[\tilde g(\nabla_{E_i}E_k,E_j)+\tilde g(E_i,\nabla_{E_j}E_k)]\Big\}+\R(E_i,E_j)+$$$$+\lambda \delta^{ij}+\mu \eta^i\eta^j=0$$
for any $(i,j)\in\{(1,1), (2,2), (3,3), (1,2), (1,3), (2,3)\}$, which is equivalent to the following system
\pagebreak
\begin{equation}\label{17}
\left\{
    \begin{aligned}
     & E_1(V^1)-aV^2-bV^3+\R(E_1,E_1)+\lambda+\mu (\eta^1)^2=0 \\
   &   E_2(V^2)-cV^1-dV^3+\R(E_2,E_2)+\lambda+\mu (\eta^2)^2=0 \\
  &    E_3(V^3)+\R(E_3,E_3)+\lambda+\mu (\eta^3)^2=0 \\
   &   \frac{\displaystyle 1}{\displaystyle 2}[E_1(V^2)+E_2(V^1)+aV^1+cV^2]+\mu \eta^1\eta^2=0 \\
   &   \frac{\displaystyle 1}{\displaystyle 2}[E_1(V^3)+E_3(V^1)+bV^1]+\R(E_1,E_3)+\mu \eta^1\eta^3=0 \\
  &    \frac{\displaystyle 1}{\displaystyle 2}[E_2(V^3)+E_3(V^2)+dV^2]+\R(E_2,E_3)+\mu \eta^2\eta^3=0
    \end{aligned}
  \right..
\end{equation}

We shall further consider the cases when the potential vector field of the almost $\eta$-Ricci soliton $(I,\tilde g, \lambda,\mu)$ is $V=\frac{\displaystyle \partial}{\displaystyle \partial x^3}$, or when $\eta=dx^3$.

\subsection{Almost $\eta$-Ricci solitons with $V=\frac{\displaystyle \partial}{\displaystyle \partial x^3}$}

If $V^1=V^2=0$ and $V^3=1$, then the system (\ref{17}) becomes
\begin{equation}\label{18}
\left\{
    \begin{aligned}
     & -b+\R(E_1,E_1)+\lambda+\mu (\eta^1)^2=0 \\
     & -d+\R(E_2,E_2)+\lambda+\mu (\eta^2)^2=0 \\
     & \R(E_3,E_3)+\lambda+\mu (\eta^3)^2=0 \\
     & \mu \eta^1\eta^2=0 \\
     & \R(E_1,E_3)+\mu \eta^1\eta^3=0 \\
     &\R(E_2,E_3)+\mu \eta^2\eta^3=0
    \end{aligned}
  \right..
\end{equation}

\begin{proposition}
Let $(I, \tilde g, V,\lambda, \mu)$ be an almost $\eta$-Ricci soliton with $V=\frac{\displaystyle \partial}{\displaystyle \partial x^3}$. If $f_i=\nolinebreak f_i(x^i)$ for $i\in \{1,2\}$ and $\mu$ is nowhere zero on $I$, then $\eta=0$ and $V$ is a Killing vector field.
\end{proposition}
\begin{proof}
In this case, we have $a=b=c=d=0$, and $\R(E_i,E_j)=0$ for any $i,j\in \{1,2,3\}$, and (\ref{18})
becomes
\begin{equation}\label{18a}
\left\{
    \begin{aligned}
     &\lambda+\mu (\eta^1)^2=0 \\
     & \lambda+\mu (\eta^2)^2=0 \\
     & \lambda+\mu (\eta^3)^2=0 \\
     & \mu \eta^1\eta^2=\mu \eta^1\eta^3=\mu \eta^2\eta^3=0
    \end{aligned}
  \right..
\end{equation}
Since $\mu$ is nowhere zero, we get $(\eta^1)^2=(\eta^2)^2=(\eta^3)^2$, hence, $\eta^i=0$ for $i\in \{1,2,3\}$, and $\lambda=0$ from (\ref{18a}); therefore, $\pounds _{V}\tilde g=0$ by means of (\ref{11}).
\end{proof}

\begin{proposition}\label{pl1}
Let $(I,\tilde g,V,\lambda,\mu)$ be an almost $\eta$-Ricci soliton with $V=\nolinebreak \frac{\displaystyle \partial}{\displaystyle \partial x^3}$. If $f_i=f_i(x^3)$ for $i\in \{1,2\}$, $\eta^i=:f$ for $i\in \{1,2,3\}$, and $\mu$ is nowhere zero on $I$, then $\eta=0$.
Moreover,

(i) if one of the functions $f_1$ and $f_2$ is constant, then the other one is constant, too; in this case, $(I,\tilde g)$ is a flat Riemannian manifold, $\lambda=0$, and $V$ is a Killing vector field;

(ii) if $\frac{\displaystyle f_1}{\displaystyle f_2}$ is constant, then $f_i(x^3)=c_ie^{c_0e^{-x^3}}$ for $i\in \{1,2\}$, where $c_0\in \mathbb R$,
\linebreak 
$c_1,c_2\in \mathbb R\setminus\{0\}$.
\end{proposition}
\begin{proof}
In this case, we have $a=0$, $b=\frac{\displaystyle f_1'}{\displaystyle f_1}$, $c=0$, $d=\frac{\displaystyle f_2'}{\displaystyle f_2}$, and
$$\R(E_1,E_1)=b'-b^2-bd,\ \ \R(E_2,E_2)=d'-d^2-bd, \ \ \R(E_3,E_3)=b'+d'-b^2-d^2,$$
$$\R(E_1,E_2)=\R(E_1,E_3)=\R(E_2,E_3)=0,$$
and (\ref{18})
becomes
\begin{equation}\label{18ai}
\left\{
    \begin{aligned}
   & -b+b'-b^2-bd+\lambda+\mu (\eta^1)^2=0 \\
   & -d+d'-d^2-bd+\lambda+\mu (\eta^2)^2=0 \\
   &   b'+d'-b^2-d^2+\lambda+\mu (\eta^3)^2=0 \\
   &   \mu \eta^1\eta^2=\mu \eta^1\eta^3=\mu \eta^2\eta^3=0
    \end{aligned}
  \right..
\end{equation}
Since $\mu$ is nowhere zero, (\ref{18ai}) is equivalent to
\begin{equation}\label{18vb}
\left\{
    \begin{aligned}
     & b(d+1)=d^2-d' \\
     & d(b+1)=b^2-b' \\
     & \lambda=b^2-b'+d^2-d'\\
& f^2=0
    \end{aligned}
  \right.,
\end{equation}
which implies $\eta=0$.

(i) $f_1$ is constant if and only if $b=0$, and from (\ref{18vb}), we deduce that $b=0$ if and only if $d=0$, i.e., if and only if $f_2$ is constant. In this case, $R=0$, $\lambda=0$, and $\pounds_V\tilde g=0$.

(ii) If $\frac{\displaystyle f_1}{\displaystyle f_2}$ is constant, then $b=d$, and from (\ref{18vb}), we get
\begin{equation}\label{18vbn}
\left\{
    \begin{aligned}
    &  b'=-b \\
    &  \lambda=2(b^2-b')
    \end{aligned}
  \right..
\end{equation}
If $b=d=0$, then $f_1$ and $f_2$ are constant and $\lambda =0$.
If $b=d\neq 0$ (hence, if they are nowhere zero), then
\begin{equation}\label{18vbm}
\left\{
    \begin{aligned}
    &  \frac{\displaystyle b'}{\displaystyle b}=-1 \\
    &  \lambda=2(b^2-b')
    \end{aligned}
  \right.,
\end{equation}
from where we get
$$\frac{f_i'(x^3)}{f_i(x^3)}=b(x^3)=c_0e^{-x^3}, \ \ c_0\in \mathbb R\setminus \{0\}, i\in \{1,2\},$$
which, by integration, gives
$f_i(x^3)=c_ie^{-c_0e^{-x^3}}$ for $i\in \{1,2\}$, where $c_1,c_2\in \mathbb R\setminus \{0\}$.
\end{proof}

\begin{corollary}
Under the hypotheses of Proposition \ref{pl1}, if $f_1=f_2=:\nolinebreak f(x^3)$, then $f(x^3)=c_1e^{c_2e^{-x^3}}$, where $c_1\in \mathbb R\setminus\{0\}$, $c_2\in \mathbb R$.
\end{corollary}
\begin{proof}
It follows from Proposition \ref{pl1} (ii).
\end{proof}

\begin{proposition}
Let $(I,\tilde g,V,\lambda,\mu)$ be an almost $\eta$-Ricci soliton with $V=\frac{\displaystyle \partial}{\displaystyle \partial x^3}$. If $f_1=f_1(x^1)$, $f_2=f_2(x^3)$, $\eta^i=:f$ for $i\in \{1,2,3\}$, and $\mu$ is nowhere zero on $I$, then $\eta=0$ and $\lambda=0$. Moreover, $f_2$ is constant and $V$ is a Killing vector field.
\end{proposition}
\begin{proof}
In this case, we have $a=b=c=0$, $d=\frac{\displaystyle f_2'}{\displaystyle f_2}$, and
$$\R(E_2,E_2)=\R(E_3,E_3)=d'-d^2,$$
$$\R(E_1,E_1)=\R(E_1,E_2)=\R(E_1,E_3)=\R(E_2,E_3)=0,$$
and (\ref{18}) becomes
\begin{equation}\label{18aia}
\left\{
    \begin{aligned}
   &   \lambda+\mu (\eta^1)^2=0 \\
  &    -d+d'-d^2+\lambda+\mu (\eta^2)^2=0 \\
  &    d'-d^2+\lambda+\mu (\eta^3)^2=0 \\
 &     \mu \eta^1\eta^2=\mu \eta^1\eta^3=\mu \eta^2\eta^3=0
    \end{aligned}
  \right..
\end{equation}
Since $\mu$ is nowhere zero, (\ref{18aia}) is equivalent to
\begin{equation*}
\left\{
    \begin{aligned}
    &  d=d'-d^2=0 \\
   &   \lambda=-\mu f^2 \\
   &   f^2=0
    \end{aligned}
  \right..
\end{equation*}
It follows that $\eta=0$, $\lambda=0$, and $\pounds_V\tilde g=0$.
\end{proof}

\begin{proposition}\label{pl2}
Let $(I,\tilde g,V,\lambda,\mu)$ be an almost $\eta$-Ricci soliton with $V=\nolinebreak \frac{\displaystyle \partial}{\displaystyle \partial x^3}$. If $f_i=f_i(x^2)$ for $i\in\{1,2\}$, $\eta^i=:f$ for $i\in \{1,2,3\}$, and $\mu$ is nowhere zero on $I$, then $\eta=0$ and $\lambda=0$. Moreover, $f_1$ is constant, or $f_1'$ is nowhere zero and $f_2=c_0\frac{\displaystyle f_1^2}{\displaystyle f_1'}$, where $c_0\in \mathbb R\setminus \{0\}$.
\end{proposition}
\begin{proof}
In this case, we have $a=f_2\frac{\displaystyle f_1'}{\displaystyle f_1}$, $b=c=d=0$, and
$$\R(E_1,E_1)=\R(E_2,E_2)=E_2(a)-a^2,$$
$$\R(E_1,E_2)=\R(E_1,E_3)=\R(E_2,E_3)=\R(E_3,E_3)=0,$$
and (\ref{18}) becomes
\begin{equation}\label{18v}
\left\{
    \begin{aligned}
  &    f_2f_2'\frac{\displaystyle f_1'}{\displaystyle f_1}+\frac{\displaystyle f_2^2}{\displaystyle f_1^2}[f_1''f_1-2(f_1')^2]+\lambda+\mu (\eta^1)^2=0 \\
 &     f_2f_2'\frac{\displaystyle f_1'}{\displaystyle f_1}+\frac{\displaystyle f_2^2}{\displaystyle f_1^2}[f_1''f_1-2(f_1')^2]+\lambda+\mu (\eta^2)^2=0 \\
 &     \lambda+\mu (\eta^3)^2=0 \\
  &    \mu \eta^1\eta^2=\mu \eta^1\eta^3=\mu \eta^2\eta^3=0
    \end{aligned}
  \right..
\end{equation}
Since $\mu$ is nowhere zero, (\ref{18v}) is equivalent to
\begin{equation*}
\left\{
    \begin{aligned}
  &    f_2f_2'\frac{\displaystyle f_1'}{\displaystyle f_1}+\frac{\displaystyle f_2^2}{\displaystyle f_1^2}[f_1''f_1-2(f_1')^2]+\lambda+\mu f^2=0 \\
&      \lambda=-\mu f^2 \\
 &     f^2=0
    \end{aligned}
  \right..
\end{equation*}
It follows that $\eta=0$, $\lambda=0$, and, from the first equation of the previous system, we get
$$\frac{\displaystyle f_1'}{\displaystyle f_1}\frac{\displaystyle f_2'}{\displaystyle f_2}+\frac{\displaystyle f_1''}{\displaystyle f_1}=2\left(\frac{\displaystyle f_1'}{\displaystyle f_1}\right)^2.$$
From Proposition \ref{ps11ax}, we deduce that $f_1$ is constant, or $f_1'$ is nowhere zero and $f_2=c_0\frac{\displaystyle f_1^2}{\displaystyle f_1'}$, where $c_0\in \mathbb R\setminus \{0\}$.
\end{proof}

\begin{corollary}
Under the hypotheses of Proposition \ref{pl2}, if $f_1=f_2=:\nolinebreak f(x^2)$, then $f(x^2)=c_1e^{c_2x^2}$, where $c_1\in \mathbb R\setminus\{0\}$, $c_2\in \mathbb R$.
\end{corollary}
\begin{proof}
It follows from Proposition \ref{pl2} that $ff''-(f')^2=0$, that is, $\left(\frac{\displaystyle f'}{\displaystyle f}\right)'=\nolinebreak 0$, with the solution
$f(x^2)=c_1e^{c_2x^2}$, where $c_1\in \mathbb R\setminus\{0\}$, $c_2\in\mathbb R$.
\end{proof}

\begin{proposition}\label{pl3}
Let $(I,\tilde g,V,\lambda,\mu)$ be an almost $\eta$-Ricci soliton with $V=\nolinebreak \frac{\displaystyle \partial}{\displaystyle \partial x^3}$. If $f_1=f_1(x^2)$, $f_2=f_2(x^1)$, $\eta^i=:f$ for $i\in \{1,2,3\}$, and $\mu$ is nowhere zero on $I$, then $\eta=0$ and $\lambda=0$. Moreover, $$\frac{\displaystyle f_1''f_1-2(f_1')^2}{\displaystyle f_1^4}=-\frac{\displaystyle  f_2''f_2-2(f_2')^2}{\displaystyle f_2^4} \ =\textit{constant}.$$
\end{proposition}
\begin{proof}
In this case, we have $a=f_2\frac{\displaystyle f_1'}{\displaystyle f_1}$, $b=0$, $c=f_1\frac{\displaystyle f_2'}{\displaystyle f_2}$, $d=0$, and
$$\R(E_1,E_1)=\R(E_2,E_2)=E_1(c)+E_2(a)-a^2-c^2,$$
$$\R(E_1,E_2)=\R(E_1,E_3)=\R(E_2,E_3)=\R(E_3,E_3)=0,$$
and (\ref{18}) becomes
\begin{equation}\label{18m}
\left\{
    \begin{aligned}
&   f_1 \frac{\displaystyle \partial c}{\displaystyle \partial x^1}+f_2 \frac{\displaystyle \partial a}{\displaystyle \partial x^2}-a^2-c^2+\lambda+\mu (\eta^1)^2=0 \\
&      f_1 \frac{\displaystyle \partial c}{\displaystyle \partial x^1}+f_2 \frac{\displaystyle \partial a}{\displaystyle \partial x^2}-a^2-c^2+\lambda+\mu (\eta^2)^2=0 \\
&      \lambda+\mu (\eta^3)^2=0 \\
 &     \mu \eta^1\eta^2=\mu \eta^1\eta^3=\mu \eta^2\eta^3=0
    \end{aligned}
  \right..
\end{equation}
Since $\mu$ is nowhere zero and $\eta^i=f$, (\ref{18m}) is equivalent to
\begin{equation*}
\left\{
    \begin{aligned}
 &      f_1 \frac{\displaystyle \partial c}{\displaystyle \partial x^1}+f_2 \frac{\displaystyle \partial a}{\displaystyle \partial x^2}-a^2-c^2+\lambda+\mu f^2=0 \\
&      \lambda=-\mu f^2 \\
&      f^2=0
    \end{aligned}
  \right..
\end{equation*}
It follows that $\eta=0$, $\lambda=0$, and, from the first equation of the previous system, we get the relation between $f_1$ and $f_2$. Since $f_1$ depends only on $x^2$ and $f_2$ depends only on $x^1$, we deduce that the obtained ratio must be a constant.
\end{proof}

\begin{proposition}\label{pl3a}
Let $(I,\tilde g,V,\lambda,\mu)$ be an almost $\eta$-Ricci soliton with $V=\frac{\displaystyle \partial}{\displaystyle \partial x^3}$. If $f_1=f_1(x^2)$, $f_2=f_2(x^3)$, $\eta^i=:f$ for $i\in \{1,2,3\}$, and $\mu$ is nowhere zero on $I$, then $\eta=0$ and $\lambda=0$. Moreover, $V$ is a Killing vector field, $f_2$ is constant, and either $f_1$ is constant or $f_1(x^2)=\displaystyle\frac{c_1}{x^2-c_2}$, where $c_1\in \mathbb{R}\setminus\{0\}$, $c_2\in \mathbb{R}\setminus I_2$.
\end{proposition}
\begin{proof}
In this case, we have $a=f_2\frac{\displaystyle f_1'}{\displaystyle f_1}$, $b=c=0$, $d=\frac{\displaystyle f_2'}{\displaystyle f_2}$, and
$$\R(E_1,E_1)=E_2(a)-a^2,\ \
\R(E_2,E_2)=E_2(a)-a^2+E_3(d)-d^2, \ \ \R(E_3,E_3)=E_3(d)-d^2,$$
$$\R(E_1,E_2)=\R(E_1,E_3)=0, \ \ \R(E_2,E_3)=E_3(a),$$
and (\ref{18}) becomes
\pagebreak
\begin{equation}\label{18ma}
\left\{
    \begin{aligned}
&   f_2 \frac{\displaystyle \partial a}{\displaystyle \partial x^2}-a^2+\lambda+\mu (\eta^1)^2=0 \\
&  -d+f_2 \frac{\displaystyle \partial a}{\displaystyle \partial x^2}-a^2+d'-d^2+\lambda+\mu (\eta^2)^2=0 \\
&      d'-d^2+\lambda+\mu (\eta^3)^2=0 \\
 &     \mu \eta^1\eta^2=\mu \eta^1\eta^3=0\\
& \frac{\displaystyle \partial a}{\displaystyle \partial x^3}+\mu \eta^2\eta^3=0
    \end{aligned}
  \right..
\end{equation}
Since $\mu$ is nowhere zero and $\eta^i=f$, (\ref{18ma}) is equivalent to
\begin{equation*}
\left\{
    \begin{aligned}
&   f_2 \frac{\displaystyle \partial a}{\displaystyle \partial x^2}-a^2-d=0 \\
&  d=d'-d^2 \\
&    \lambda=-d \\
 &     f^2=0\\
& \frac{\displaystyle \partial a}{\displaystyle \partial x^3}=0
    \end{aligned}
  \right..
\end{equation*}
It follows that $\eta=0$ and $f_1'f_2'=0$, which, together with the second equation of the system, implies that one of the functions $f_1$ and $f_2$ must be constant. If $f_1=k_1\in\mathbb R\setminus \{0\}$, then $a=0$, $d=0$ (hence, $f_2$ is constant), and $\lambda=0$. If $f_2=k_2\in\mathbb R\setminus \{0\}$, then $d=0$, $\lambda=0$, and $$\left(\frac{\displaystyle f_1'}{\displaystyle f_1}\right)'=\left(\frac{\displaystyle f_1'}{\displaystyle f_1}\right)^2.$$
Then, either $f_1$ is a constant, too, or $f_1'\neq 0$, in which case, by integration, we obtain
$-\frac{1}{\frac{\displaystyle f_1'(x^2)}{\displaystyle f_1(x^2)}}=x^2+c_0$, i.e., $f_1(x^2)=\frac{\displaystyle c_1}{\displaystyle x^2-c_2}$, where $c_1\in \mathbb R\setminus \{0\}$, $c_2\in \mathbb R\setminus I_2$.
Since $b=d=0$, we deduce that $\pounds_V\tilde g=0$.
\end{proof}

\subsection{Almost $\eta$-Ricci solitons with $\eta=dx^3$}

If $\eta^1=\eta^2=0$ and $\eta^3=1$, then the system (\ref{17}) becomes
\begin{equation}\label{19}
\left\{
    \begin{aligned}
&      E_1(V^1)-aV^2-bV^3+\R(E_1,E_1)+\lambda=0 \\
 &     E_2(V^2)-cV^1-dV^3+\R(E_2,E_2)+\lambda=0 \\
 &     E_3(V^3)+\R(E_3,E_3)+\lambda+\mu=0 \\
 &     E_1(V^2)+E_2(V^1)+aV^1+cV^2=0 \\
&      \frac{\displaystyle 1}{\displaystyle 2}[E_1(V^3)+E_3(V^1)+bV^1]+\R(E_1,E_3)=0 \\
 &     \frac{\displaystyle 1}{\displaystyle 2}[E_2(V^3)+E_3(V^2)+dV^2]+\R(E_2,E_3)=0
    \end{aligned}
  \right..
\end{equation}

\begin{theorem}\label{nb}
Let $(I, \tilde g, V,\lambda, \mu)$ be an $\eta$-Ricci soliton with $\eta=dx^3$. If $f_i=f_i(x^i)$ for $i\in \{1,2\}$,
then
\begin{equation}\label{6262}
\left\{
    \begin{aligned}
& V^1(x^1,x^2,x^3)=-\lambda F_1(x^1)+c_1 F_2(x^2)+c_2x^3+c_3\\
& V^2(x^1,x^2,x^3)=-c_1 F_1(x^1)-\lambda F_2(x^2)+c_4x^3+c_5\\
& V^3(x^1,x^2,x^3)=-c_2F_1(x^1)-c_4F_2(x^2)-(\lambda+\mu)x^3+c_6
    \end{aligned}
  \right.,
\end{equation}
where $F_i$ is an antiderivative of $\frac{\displaystyle 1}{\displaystyle f_i}$ for $i\in\{1,2\}$ and $c_i\in \mathbb R$ for $i\in\nolinebreak \{1,2,3,4,5,6\}$.
\end{theorem}
\begin{proof}
In this case, we have $a=b=c=d=0$, and $\R(E_i,E_j)=0$ for any $i,j\in \{1,2,3\}$, and (\ref{19})
becomes
\begin{equation}\label{19b}
\left\{
    \begin{aligned}
   &   \frac{\displaystyle \partial V^1}{\displaystyle \partial x^1}=-\frac{\displaystyle \lambda}{\displaystyle f_1} \\
 &     \frac{\displaystyle \partial V^2}{\displaystyle \partial x^2}=-\frac{\displaystyle \lambda}{\displaystyle f_2}\\
 &     \frac{\displaystyle \partial V^3}{\displaystyle \partial x^3}=-(\lambda+\mu) \\
 &     f_1\frac{\displaystyle \partial V^2}{\displaystyle \partial x^1}=-f_2\frac{\displaystyle \partial V^1}{\displaystyle \partial x^2} \\
  &    f_1\frac{\displaystyle \partial V^3}{\displaystyle \partial x^1}=-\frac{\displaystyle \partial V^1}{\displaystyle \partial x^3} \\
  &    f_2\frac{\displaystyle \partial V^3}{\displaystyle \partial x^2}=-\frac{\displaystyle \partial V^2}{\displaystyle\partial x^3}
    \end{aligned}
  \right..
\end{equation}
Since $f_i$ depends only on $x^i$, from the first three equations of (\ref{19b}), we find that
\begin{equation*}
\left\{
    \begin{aligned}
& V^1(x^1,x^2,x^3)=-\lambda F_1(x^1)+h_1(x^2,x^3)\\
& V^2(x^1,x^2,x^3)=-\lambda F_2(x^2)+h_2(x^1,x^3)\\
& V^3(x^1,x^2,x^3)=-(\lambda +\mu)x^3+h_3(x^1,x^2)
\end{aligned}
  \right.,
\end{equation*}
where $F_i'=\frac{\displaystyle 1}{\displaystyle f_i}$ for $i\in\{1,2\}$. From the last three equations of (\ref{19b}), we infer:
\begin{align}
f_1(x^1)\frac{\displaystyle \partial h_2}{\displaystyle \partial x^1}(x^1,x^3)&=-f_2(x^2)\frac{\displaystyle \partial h_1}{\displaystyle \partial x^2}(x^2,x^3),\label{sdf}\\
f_1(x^1)\frac{\displaystyle \partial h_3}{\displaystyle \partial x^1}(x^1,x^2)&=-\frac{\displaystyle \partial h_1}{\displaystyle \partial x^3}(x^2,x^3),\label{12}\\
f_2(x^2)\frac{\displaystyle \partial h_3}{\displaystyle \partial x^2}(x^1,x^2)&=-\frac{\displaystyle \partial h_2}{\displaystyle \partial x^3}(x^1,x^3).\label{12a}
\end{align}
Denoting $f_2(x^2)\frac{\displaystyle \partial h_1}{\displaystyle \partial x^2}(x^2,x^3)=:\bar h_1(x^3)$, we have $$\frac{\displaystyle \partial h_1}{\displaystyle \partial x^2}(x^2,x^3)=\frac{\displaystyle \bar h_1(x^3)}{\displaystyle f_2(x^2)},$$ which, by integration, implies
$$h_1(x^2,x^3)=\bar h_1(x^3) F_2(x^2)+\hat h_1(x^3),$$
and, by replacing it in \eqref{12}, it gives
$$-\bar h_1'(x^3) F_2(x^2)-f_1(x^1)\frac{\displaystyle \partial h_3}{\displaystyle \partial x^1}(x^1,x^2)=\hat h_1'(x^3).$$
Differentiating the previous relation with respect to $x^2$, it implies
$$\bar h_1'(x^3)=-f_1(x^1)f_2(x^2)\frac{\displaystyle \partial^2 h_3}{\displaystyle \partial x^2\partial x^1}(x^1,x^2),$$
which must be a constant, let's say $c_1$. Therefore, $$\bar h_1(x^3)=c_1x^3+c_2$$ and
$$\frac{\displaystyle \partial^2 h_3}{\displaystyle \partial x^2\partial x^1}(x^1,x^2)=-\frac{\displaystyle c_1}{\displaystyle f_1(x^1)f_2(x^2)},$$ which, by double integration, implies
$$h_3(x^1,x^2)=-c_1F_1(x^1)F_2(x^2)+l_1(x^1)+l_2(x^2).$$
Also, from (\ref{sdf}), we have $\frac{\displaystyle \partial h_2}{\displaystyle \partial x^1}(x^1,x^3)=-\frac{\displaystyle \bar h_1(x^3)}{\displaystyle f_1(x^1)}$, which gives
$$h_2(x^1,x^3)=-\bar h_1(x^3) F_1(x^1)+\hat h_2(x^3).$$
From \eqref{12}, we find
$$\hat h_1'(x^3)=-f_1(x^1)l_1'(x^1),$$
which must be a constant, let's say, $c_3$,
which gives
$$\hat h_1(x^3)=c_3x^3+c_4.$$
Since $f_1(x^1)l_1'(x^1)=-c_3$, we get
$$l_1(x^1)=-c_3F_1(x^1)+c_5.$$
From \eqref{12a}, we find
$$\hat h_2'(x^3)=2c_1F_1(x^1)-f_2(x^2)l_2'(x^2),$$
which must be a constant, let's say, $c_6$,
which gives
$$\hat h_2(x^3)=c_6x^3+c_7.$$
Since $2c_1F_1(x^1)=f_2(x^2)l_2'(x^2)+c_6$ and $F_1$ is not constant, we get
$$c_1=0, \ \ l_2(x^2)=-c_6F_2(x^2)+c_8.$$
We have obtained
\begin{equation*}
\left\{
    \begin{aligned}
& V^1(x^1,x^2,x^3)=-\lambda F_1(x^1)+c_2 F_2(x^2)+c_3x^3+c_4\\
& V^2(x^1,x^2,x^3)=-\lambda F_2(x^2)-c_2 F_1(x^1)+c_6x^3+c_7\\
& V^3(x^1,x^2,x^3)=-(\lambda+\mu)x^3-c_3F_1(x^1)+c_5-c_6F_2(x^2)+c_8
\end{aligned}
  \right.,
\end{equation*}
which, by changing the indices of some constants, gives (\ref{6262}).
\end{proof}

\begin{example}
For $f_i=k_i\in \mathbb R\setminus \{0\}$, $i\in \{1,2\}$, the vector field $V$ with $V^1$, $V^2$ and $V^3$ given by:
\begin{equation*}
\left\{
    \begin{aligned}
&     V^1(x^1,x^2,x^3)=-\frac{\displaystyle \lambda}{\displaystyle k_1} x^1+\frac{\displaystyle c_1}{\displaystyle k_2} x^2+c_2x^3+c_3\\
& V^2(x^1,x^2,x^3)=-\frac{\displaystyle c_1}{\displaystyle k_1} x^1-\frac{\displaystyle \lambda}{\displaystyle k_2} x^2+c_4x^3+c_5\\
& V^3(x^1,x^2,x^3)=-\frac{\displaystyle c_2}{\displaystyle k_1} x^1-\frac{\displaystyle c_4}{\displaystyle k_2} x^2-(\lambda+\mu)x^3+c_6
    \end{aligned}
  \right.,
\end{equation*}
where $c_i\in \mathbb R$ for $i\in\{1,2,3,4,5,6\}$, is the potential vector field of the $\eta$-Ricci soliton $(I, \tilde g, \lambda,\mu)$ for $\eta=dx^3$.
\end{example}

\begin{remark}
If $f_i=f_i(x^i)$ for $i\in \{1,2\}$, then $V$ with $V^1$, $V^2$ and $V^3$ given by:
\begin{equation*}
\left\{
    \begin{aligned}
&     V^1(x^2,x^3)=c_1F_2(x^2)+c_2x^3+c_3\\
& V^2(x^1,x^3)=-c_1F_1(x^1)+c_4x^3+c_5\\
& V^3(x^1,x^2)=-c_2F_1(x^1)-c_4F_2(x^2)+c_6
    \end{aligned}
  \right.
\end{equation*}
is a Killing vector field on $(I,\tilde g)$, where $F_i$ is an antiderivative of $\frac{\displaystyle 1}{\displaystyle f_i}$ for $i\in\{1,2\}$ and $c_i\in \mathbb R$ for $i\in\{1,2,3,4,5,6\}$.
\end{remark}

\begin{theorem}\label{gs}
Let $(I, \tilde g, V,\lambda, \mu)$ be an almost $\eta$-Ricci soliton with $\eta=dx^3$. If 
\linebreak
$f_i=f_i(x^3)$ for $i\in \{1,2\}$, and $V^i=V^i(x^3)$ for $i\in \{1,2,3\}$, then:
\begin{equation*}
\left\{
    \begin{aligned}
& V^1=\frac{\displaystyle c_1}{\displaystyle f_1}\\
& V^2=\frac{\displaystyle c_2}{\displaystyle f_2}\\
& (b-d)V^3=(b-d)'-(b^2-d^2)\\
& (b-d)\lambda=b'd-bd'\\
& (b-d)^2\mu=-(b-d)[(b-d)''+b'd-bd']+[(b-d)']^2+(b-d)^2(b^2+d^2)
    \end{aligned}
  \right.,
\end{equation*}
where $c_1$, $c_2\in \mathbb R$.

We have the following cases.

(i) On any open interval $J_3\subseteq I_3$ on which $\left(\frac{\displaystyle f_1}{\displaystyle f_2}\right)'\neq 0$ everywhere (equivalent to $b\neq d$ on $J_3$), we have:
$$V^1=\frac{\displaystyle c_1}{\displaystyle f_1}, \ \
V^2=\frac{\displaystyle c_2}{\displaystyle f_2}, \ \
V^3=\frac{\displaystyle (b-d)'}{\displaystyle b-d}-(b+d),$$
$$\lambda=\frac{\displaystyle b'd-bd'}{\displaystyle b-d},\ \
\mu=-\frac{\displaystyle (b-d)''+b'd-bd'}{\displaystyle b-d}+\left(\frac{\displaystyle (b-d)'}{\displaystyle b-d}\right)^2+b^2+d^2,$$
where $c_1$, $c_2\in \mathbb R$.
\pagebreak

In particular:

\hspace{0.5cm} (1) if $f_1$ is constant on an open interval $J\subseteq J_3$ (from which $f_2'\neq 0$ everywhere on $J$), we have on $J$:
$$V^1=c_1,\ \
V^2=\frac{\displaystyle c_2}{\displaystyle f_2},\ \
V^3=\frac{\displaystyle d'-d^2}{\displaystyle d},$$
$$\lambda=0, \ \ \mu=-\frac{\displaystyle d''}{\displaystyle d}+\left(\frac{\displaystyle d'}{\displaystyle d}\right)^2+d^2,$$
where $c_1$, $c_2\in \mathbb R$;

\hspace{0.5cm} (2) if $f_2$ is constant on an open interval $J\subseteq J_3$ (from which $f_1'\neq 0$ everywhere on $J$), we have on $J$:
$$V^1=\frac{\displaystyle c_1}{\displaystyle f_1},\ \
V^2=c_2, \ \
V^3=\frac{\displaystyle b'-b^2}{\displaystyle b},$$
$$\lambda=0, \ \ \mu=-\frac{\displaystyle b''}{\displaystyle b}+\left(\frac{\displaystyle b'}{\displaystyle b}\right)^2+b^2,$$
where $c_1$, $c_2\in \mathbb R$.

(ii) On any open interval $J_3\subseteq I_3$ on which $\frac{\displaystyle f_1}{\displaystyle f_2}$ is constant (equivalent to $b=d$ on $J_3$), we have:
$$V^1=\frac{\displaystyle c_1}{\displaystyle f_1}, \ \ V^2=\frac{\displaystyle c_2}{\displaystyle f_2},\ \  b V^3=b'-2b^2+\lambda, \ \ (V^3)'=-2(b'-b^2)-(\lambda+\mu),$$
where $c_1$, $c_2\in \mathbb R$.

In addition:

\hspace{0.5cm} (1) if $J\subseteq J_3$ is an open interval with $f_1'\neq 0$ (and $f_2'\neq 0$) everywhere on $J$, we have on $J$:
$$V^1=\frac{\displaystyle c_1}{\displaystyle f_1}, \ \ V^2=\frac{\displaystyle c_2}{\displaystyle f_2},\ \  V^3=\frac{\displaystyle b'-2b^2+\lambda}{\displaystyle b},$$
$$\lambda=-\left(\frac{\displaystyle b'+\lambda}{\displaystyle b}\right)'+2b^2-\mu,$$
where $c_1$, $c_2\in \mathbb R$;

\hspace{0.5cm} (2) if $f_1$ and $f_2$ are constant on an open interval $J\subseteq J_3$, we have on $J$:
$$V^1=c_1, \ \ V^2=c_2,\ \
V^3=F,$$
$$\lambda=0,$$
where $F'=-\mu$ and $c_1$, $c_2\in \mathbb R$.
\end{theorem}
\begin{proof}
In this case, we have $a=0$, $b=\frac{\displaystyle f_1'}{\displaystyle f_1}$, $c=0$, $d=\frac{\displaystyle f_2'}{\displaystyle f_2}$, and
$$\R(E_1,E_1)=b'-b^2-bd,\ \ \R(E_2,E_2)=d'-d^2-bd, \ \ \R(E_3,E_3)=b'+d'-b^2-d^2,$$
$$\R(E_1,E_2)=\R(E_1,E_3)=\R(E_2,E_3)=0,$$
and (\ref{19}) becomes
\begin{equation*}
\left\{
    \begin{aligned}
 &      f_1\frac{\displaystyle \partial V^1}{\displaystyle \partial x^1}-bV^3+b'-b^2-bd+\lambda=0 \\
 &      f_2\frac{\displaystyle \partial V^2}{\displaystyle \partial x^2}-dV^3+d'-d^2-bd+\lambda=0\\
&      \frac{\displaystyle \partial V^3}{\displaystyle \partial x^3}+b'-b^2+d'-d^2+\lambda+\mu=0\\
  &     f_1\frac{\displaystyle \partial V^2}{\displaystyle \partial x^1}+f_2\frac{\displaystyle \partial V^1}{\displaystyle \partial x^2}=0 \\
  &     f_1\frac{\displaystyle \partial V^3}{\displaystyle \partial x^1}+\frac{\displaystyle \partial V^1}{\displaystyle \partial x^3}+bV^1=0 \\
  &     f_2\frac{\displaystyle \partial V^3}{\displaystyle \partial x^2}+\frac{\displaystyle \partial V^2}{\displaystyle \partial x^3}+dV^2=0
    \end{aligned}
  \right.,
\end{equation*}
which is equivalent to
\begin{equation}\label{3i}
\left\{
    \begin{aligned}
& (V^1)'=-bV^1\\
& (V^2)'=-dV^2\\
& (V^3)'=-(b'-b^2+d'-d^2)-(\lambda+\mu)\\
& b V^3=\displaystyle b'-b^2-bd+\lambda\\
& d V^3=\displaystyle d'-d^2-bd+\lambda
    \end{aligned}
  \right..
\end{equation}
From the first two equations, we get either $V^i=0$ or $\frac{\displaystyle (V^i)'}{\displaystyle V^i}=-\frac{\displaystyle f_i'}{\displaystyle f_i}$, i.e., $V^i=\frac{\displaystyle c_i}{\displaystyle f_i}$, where $c_i\in \mathbb R  \setminus \{0\}$ for $i\in \{1,2\}$.
From the last two equations, we obtain
$$(b-d)V^3=\displaystyle (b-d)'-(b^2-d^2)$$
and
$$(b-d)\lambda=b'd-bd'.$$
Differentiating the second-last equality, multiplying the result with $b-d$ and substituting in it the expressions of $(b-d)V^3$ and $(V^3)'$, we get
$$(b-d)^2\mu=-(b-d)[(b-d)''+b'd-bd']+[(b-d)']^2+(b-d)^2(b^2+d^2).$$

(i) For $x^3\in I_3$, we notice that $\left(\frac{\displaystyle f_1}{\displaystyle f_2}\right)'(x^3)\neq 0$ if and only if $b(x^3)\neq d(x^3)$.
From the last three equalities from above, on any open interval $J_3\subseteq I_3$ such that $\left(\frac{\displaystyle f_1}{\displaystyle f_2}\right)'(x^3)\neq 0$ for any $x^3\in J_3$ (equivalent to $b\neq d$ everywhere on $J_3$), we get the expressions of $V^3$, $\lambda$, and $\mu$.

(ii) We notice that $\frac{\displaystyle f_1}{\displaystyle f_2}$ is constant on $J_3$ if and only if $\left(\frac{\displaystyle f_1}{\displaystyle f_2}\right)'=0$ on $J_3$ (equivalent to $b=d$ on $J_3$).
In this case, (\ref{3i}) becomes
\begin{equation*}
\left\{
    \begin{aligned}
& (V^1)'=-bV^1\\
& (V^2)'=-bV^2\\
& (V^3)'=-2(b'-b^2)-(\lambda+\mu)\\
& b V^3={\displaystyle b'-2b^2+\lambda}
    \end{aligned}
  \right..
\end{equation*}

On any open interval $J\subseteq J_3$ where $f_1'\neq 0$ everywhere, we have $b\neq 0$ and
$$V^3=\frac{\displaystyle b'-2b^2+\lambda}{\displaystyle b}, $$
which, by differentiation and using the third relation, gives the expression of $\mu$.

If $f_1$ and $f_2$ are constant on $J\subseteq J_3$, then $b=d=0$ on $J$ and (\ref{3i}) becomes
\begin{equation*}
\left\{
    \begin{aligned}
& \lambda=0\\
& (V^1)'=0\\
& (V^2)'=0\\
& (V^3)'=-\mu
    \end{aligned}
  \right. .\qquad \qed
\end{equation*}
\renewcommand{\qed}{}
\end{proof}

\begin{example}
The Riemannian $Sol_3$ Lie group $(\mathbb R^3,\tilde g)$, where $f_1(x^3)=e^{-x^3}$ and $f_2(x^3)=e^{x^3}$, is an $\eta$-Ricci soliton for $\eta=dx^3$, with $$V=c_1\frac{\displaystyle \partial }{\displaystyle \partial x^1}+c_2\frac{\displaystyle \partial }{\displaystyle \partial x^2} \ \ (c_1,c_2\in \mathbb R), \ \ \lambda=0, \ \ \mu=2.$$
\end{example}

\begin{remark}
If $f_i=f_i(x^3)$ for $i\in \{1,2\}$, then $V$ with $V^1$, $V^2$ and $V^3$ given by:
\begin{equation*}
\left\{
    \begin{aligned}
&     V^1(x^3)=\frac{\displaystyle c_1}{\displaystyle f_1(x^3)}\\
& V^2(x^3)=\frac{\displaystyle c_2}{\displaystyle f_2(x^3)}\\
& V^3=0
    \end{aligned}
  \right.
\end{equation*}
is a Killing vector field on $(I,\tilde g)$, where $c_i\in \mathbb R$ for $i\in\{1,2\}$.
\end{remark}

\begin{theorem}\label{gss}
Let $(I, \tilde g, V,\lambda, \mu)$ be an almost $\eta$-Ricci soliton with $\eta=dx^3$. If
\linebreak
$f_1=f_1(x^1)$, $f_2=f_2(x^3)$, and $V^i=V^i(x^3)$ for $i\in \{1,2,3\}$, then:
\begin{equation*}
\left\{
    \begin{aligned}
& V^1=c_1 \\
& V^2=\frac{\displaystyle c_2}{\displaystyle f_2}\\
& d V^3=d'-d^2\\
& \lambda=0\\
& d^2 \mu=(d')^2+d^4-dd''
    \end{aligned}
  \right.,
\end{equation*}
where $c_1,c_2\in \mathbb{R}$.

We have the following cases.

(i) On any open interval $J_3\subseteq I_3$ on which $f_2$ is constant, we have:
$$V^1=c_1, \ \ V^2=c_2, \ \
V^3=F,$$
$$\lambda=0,$$
where $F'=-\mu$ and $c_1$, $c_2\in \mathbb R$.

(ii) On any open interval $J_3\subseteq I_3$ on which $f_2'\neq 0$ everywhere (equivalent to $d\neq 0$ on $J_3$), we have:
$$V^1=c_1, \ \ V^2=\frac{\displaystyle c_2}{\displaystyle f_2}, \ \
V^3=\frac{\displaystyle d'-d^2}{\displaystyle d},$$
$$\lambda=0, \ \ \mu=-\frac{\displaystyle d''}{\displaystyle d}+\left(\frac{\displaystyle d'}{\displaystyle d}\right)^2+d^2,$$
where $c_1, c_2\in \mathbb R$.
\end{theorem}

\begin{proof}
In this case, we have $a=b=c=0$, $d=\frac{\displaystyle f_2'}{\displaystyle f_2}$, and
$$\R(E_2,E_2)=\R(E_3,E_3)=d'-d^2,$$
$$\R(E_1,E_1)=\R(E_1,E_2)=\R(E_1,E_3)=\R(E_2,E_3)=0,$$
and (\ref{19}) becomes
\pagebreak
\begin{equation*}
\left\{
    \begin{aligned}
 &      f_1\frac{\displaystyle \partial V^1}{\displaystyle \partial x^1}+\lambda=0 \\
&       f_2\frac{\displaystyle \partial V^2}{\displaystyle \partial x^2}-dV^3+d'-d^2+\lambda=0\\
&      \frac{\displaystyle \partial V^3}{\displaystyle \partial x^3}+d'-d^2+\lambda+\mu=0\\
&       f_1\frac{\displaystyle \partial V^2}{\displaystyle \partial x^1}+f_2\frac{\displaystyle \partial V^1}{\displaystyle \partial x^2}=0 \\
&       f_1\frac{\displaystyle \partial V^3}{\displaystyle \partial x^1}+\frac{\displaystyle \partial V^1}{\displaystyle \partial x^3}=0 \\
&       f_2\frac{\displaystyle \partial V^3}{\displaystyle \partial x^2}+\frac{\displaystyle \partial V^2}{\displaystyle \partial x^3}+dV^2=0
    \end{aligned}
  \right.,
\end{equation*}
equivalent to
\begin{equation}\label{19as}
\left\{
    \begin{aligned}
& \lambda=0\\
& (V^1)'=0\\
& (V^2)'=-dV^2\\
& (V^3)'=-d'+d^2-\mu\\
& d V^3=d'-d^2
    \end{aligned}
  \right.,
\end{equation}
from which we get
$$V^1=c_1, \ \ V^2=\frac{\displaystyle c_2}{\displaystyle f_2},$$
where $c_1,c_2\in \mathbb{R}$. Differentiating the last equality of the previous system, multiplying the result with $d$ and substituting in it the expressions of $dV^3$ and $(V^3)'$, we get
$$d^2 \mu=(d')^2+d^4-dd''.$$

(i) We notice that $f_2$ is constant on $J_3$ if and only if $d=0$ on $J_3$, and (\ref{19as}) becomes
\begin{equation*}
\left\{
    \begin{aligned}
&  \lambda=0\\
& (V^1)'=0\\
& (V^2)'=0\\
& (V^3)'=-\mu
    \end{aligned}
  \right..
\end{equation*}

(ii) We notice that $f_2'\neq 0$ on $J_3$ if and only if $d\neq 0$ on $J_3$. Dividing the relations for $dV^3$ and $d^2\mu$ by $d$ and $d^2$, respectively, we obtain the expressions of $V^3$ and $\mu$.
\end{proof}

\begin{example}
For $f_1(x^1)=e^{x^1}$, $f_2(x^3)=e^{x^3}$, the vector field
$$V=c_1e^{x^1}\frac{\partial }{\partial x^1}+c_2\frac{\partial }{\partial x^2}-\frac{\partial }{\partial x^3} \ \ (c_1,c_2\in \mathbb R)$$ is the potential vector field of the $\eta$-Ricci soliton $(\mathbb R^3, \tilde g, 0,1)$ for $\eta=dx^3$.
\end{example}

\begin{theorem}\label{gsc}
Let $(I, \tilde g, V,\lambda, \mu)$ be an almost $\eta$-Ricci soliton with $\eta=dx^3$. If
\linebreak
$f_i=f_i(x^2)$ for $i\in\{1,2\}$ and $V^i=V^i(x^3)$ for $i\in \{1,2,3\}$, then:
\begin{equation*}
\left\{
\begin{aligned}
&V^1=c_1 \\
&V^2=c_2 \\
&V^3=F\\
&c_1f_1'=0=c_2f_1'\\
&\lambda=2f_2^2\left(\frac{f_1'}{f_1}\right)^2-f_2f_2'\frac{f_1'}{f_1}-f_2^2\frac{f_1''}{f_1}
\end{aligned}
  \right.,
\end{equation*}
where $c_1,c_2\in \mathbb R$, $F'=G$, $\mu(x^2,x^3)=-\lambda(x^2)-G(x^3)$,
with $G=G(x^3)$ a smooth function on $I_3$.

We have the following cases.

(i) For $J_2\subseteq I_2$ a nontrivial interval on which $f_1$ is constant, we have
$$\lambda(x^2)=0,\ \ F'(x^3)=-\mu(x^2,x^3)\ \ \textrm{for }x^2\in J_2, x^3\in I_3.$$

(ii) For $f_1$ constant, we have
$$\lambda=0,\ \ \mu =\mu(x^3),\ \ F'(x^3)=-\mu(x^3)\ \ \textrm{for } x^3\in I_3.$$

(iii) For $f_1$ not constant, we have $$V^1=V^2=0.$$
\end{theorem}
\begin{proof}
In this case, we have $a=f_2\frac{\displaystyle f_1'}{\displaystyle f_1}$ , $b=c=d=0$, and
$$\R(E_1,E_1)=\R(E_2,E_2)=E_2(a)-a^2,$$
$$\R(E_1,E_2)=\R(E_1,E_3)=\R(E_2,E_3)=\R(E_3,E_3)=0,$$
and (\ref{19}) becomes
\pagebreak
\begin{equation*}\label{19ac}
\left\{
    \begin{aligned}
 &      f_1\frac{\displaystyle \partial V^1}{\displaystyle \partial x^1}-aV^2+f_2a'-a^2+\lambda=0 \\
&       f_2\frac{\displaystyle \partial V^2}{\displaystyle \partial x^2}+f_2a'-a^2+\lambda=0\\
&      \frac{\displaystyle \partial V^3}{\displaystyle \partial x^3}+\lambda+\mu=0\\
&       f_1\frac{\displaystyle \partial V^2}{\displaystyle \partial x^1}+f_2\frac{\displaystyle \partial V^1}{\displaystyle \partial x^2}+aV^1=0 \\
&       f_1\frac{\displaystyle \partial V^3}{\displaystyle \partial x^1}+\frac{\displaystyle \partial V^1}{\displaystyle \partial x^3}=0 \\
&       f_2\frac{\displaystyle \partial V^3}{\displaystyle \partial x^2}+\frac{\displaystyle \partial V^2}{\displaystyle \partial x^3}=0
    \end{aligned}
  \right.,
\end{equation*}
that is,
\begin{equation*}
\left\{
\begin{aligned}
&aV^1=0 \\
&aV^2=0 \\
&(V^1)'=0 \\
&(V^2)'=0\\
&(V^3)'=-(\lambda+\mu)\\
&f_2a'-a^2+\lambda=0
\end{aligned}
  \right..
\end{equation*}
From the second-last equation, we deduce that $\lambda+\mu$ depends only on $x^3$; therefore,
$$0=\frac{\partial (\lambda+\mu)}{\partial x^2}=\lambda'+\frac{\partial \mu}{\partial x^2}\ ,$$
which, by integration, gives
$$\mu(x^2,x^3)=-\lambda(x^2)-G(x^3)$$
with $G=G(x^3)$ a smooth function on $I_3$. We obtain
\begin{equation}\label{20}
\left\{
\begin{aligned}
&V^1=c_1 \\
&V^2=c_2 \\
&V^3=F\\
&aV^1=0 \\
&aV^2=0 \\
&\lambda=-f_2a'+a^2\\
\end{aligned}
  \right.,
\end{equation}
where $F'=G$ and $c_1, c_2\in \mathbb{R}$.

(i) We notice that $f_1$ is constant on $J_2$ if and only if $a=0$ on $J_2$; therefore, $a'=0$ and $\lambda =0$, hence $F'=-\mu$ on $J_2$.

(ii) It follows from (i) for $J_2=I_2$.

(iii) For $f_1$ not constant, there exists $x_0^2\in I_2$ such that $f_1'(x_0^2)\neq 0$, i.e., $a(x_0^2)\neq 0$, and from the second- and third-last equations of (\ref{20}), it follows that $V_1=V_2=0$.
\end{proof}

\begin{example}
The manifold $\mathbb H^2\times\mathbb R:=\{(x^1,x^2,x^3)\in\mathbb R^3 \ : \ x^2>0 \}$ with $$\tilde g=\nolinebreak \frac{\displaystyle 1}{\displaystyle f^2}dx^1\otimes dx^1+\frac{\displaystyle 1}{\displaystyle f^2}dx^2\otimes dx^2+dx^3\otimes dx^3,$$ where $f(x^2)=x^2$, is an $\eta$-Ricci soliton for $\eta=dx^3$, with
$$V=\frac{\displaystyle \partial}{\displaystyle \partial x^3}, \ \ \lambda=1, \ \ \mu=-1.$$
\end{example}

\begin{theorem}
Let $(I, \tilde g, V,\lambda, \mu)$ be an almost $\eta$-Ricci soliton with $\eta=dx^3$. If
\linebreak
$f_1=f_1(x^2)$, $f_2=f_2(x^1)$, and $V^i=V^i(x^3)$ for $i\in \{1,2,3\}$, then:
\begin{equation*}
\left\{
    \begin{aligned}
&  V^1=c_1 \\
&  V^2=c_2 \\
&  V^3=F \\
&c_1 f_2\frac{\displaystyle f_1'}{\displaystyle f_1}=-c_2 f_1\frac{\displaystyle f_2'}{\displaystyle f_2}\\
&c_2 f_2\frac{\displaystyle f_1'}{\displaystyle f_1}=c_1 f_1\frac{\displaystyle f_2'}{\displaystyle f_2}\\
&  \lambda=c_1f_1\frac{\displaystyle f_2'}{\displaystyle f_2}+f_1^2\left[\left(\frac{\displaystyle f_2'}{\displaystyle f_2}\right)^2-\left(\frac{\displaystyle f_2'}{\displaystyle f_2}\right)'\right]+f_2^2\left[\left(\frac{\displaystyle f_1'}{\displaystyle f_1}\right)^2-\left(\frac{\displaystyle  f_1'}{\displaystyle f_1}\right)'\right]
    \end{aligned}
  \right.,
\end{equation*}
where $c_1,c_2\in \mathbb R$ and $F'=-(\lambda+\mu)$.

For $c_1\neq 0$ or $c_2\neq 0$, we get $f_1$ and $f_2$ constant, $\lambda=0$, and $F'=-\mu$.
\end{theorem}
\begin{proof}
In this case, we have $a=f_2\frac{\displaystyle f_1'}{\displaystyle f_1}$, $b=0$, $c=f_1\frac{\displaystyle f_2'}{\displaystyle f_2}$, $d=0$, and
$$\R(E_1,E_1)=\R(E_2,E_2)=E_1(c)+E_2(a)-a^2-c^2,$$
$$\R(E_1,E_2)=\R(E_1,E_3)=\R(E_2,E_3)=\R(E_3,E_3)=0,$$
and (\ref{19}) becomes
\pagebreak
\begin{equation}\label{19h}
\left\{
    \begin{aligned}
 &      E_1(V^1)-aV^2+E_1(c)+E_2(a)-a^2-c^2+\lambda=0 \\
 &      E_2(V^2)-cV^1+E_1(c)+E_2(a)-a^2-c^2+\lambda=0 \\
 &      E_3(V^3)+\lambda+\mu=0 \\
 &     E_1(V^2)+E_2(V^1)+aV^1+cV^2=0 \\
 &      E_1(V^3)+E_3(V^1)=0 \\
 &      E_2(V^3)+E_3(V^2)=0
    \end{aligned}
  \right..
\end{equation}
Since $V^i$ for $i\in \{1,2,3\}$ depends only on $x^3$, (\ref{19h}) becomes
\begin{equation*}
\left\{
    \begin{aligned}
& aV^2=f_1\frac{\displaystyle \partial c}{\displaystyle \partial x^1}+f_2\frac{\displaystyle \partial a}{\displaystyle \partial x^2}-a^2-c^2+\lambda=cV^1 \\
& aV^1=-cV^2 \\
 &      (V^1)'=0 \\
 &     (V^2)'=0 \\
& (V^3)'=-(\lambda+\mu)
    \end{aligned}
  \right..
\end{equation*}
From the last three equations, we infer
$V^1=c_1$, $V^2=c_2$, where $c_1,c_2\in \mathbb R$, and $V^3=F$, where $F'=-(\lambda+\mu)$, and the previous system is equivalent to
\begin{equation*}
\left\{
    \begin{aligned}
&  V^1=c_1 \\
&  V^2=c_2 \\
&  \lambda=c_1f_1\frac{\displaystyle f_2'}{\displaystyle f_2}+f_1^2\left[\left(\frac{\displaystyle f_2'}{\displaystyle f_2}\right)^2- \left(\frac{\displaystyle f_2'}{\displaystyle f_2}\right)'\right]+f_2^2\left[\left(\frac{\displaystyle f_1'}{\displaystyle f_1}\right)^2-\left(\frac{\displaystyle  f_1'}{\displaystyle f_1}\right)'\right] \\
&  c_2a=c_1c \\
&  c_1a=-c_2c \\
&  V^3=F
    \end{aligned}
  \right..
\end{equation*}
We get $(c_1^2+c_2^2)a=0$. If $c_1\neq 0$ or $c_2\neq 0$, then $a=0$, i.e., $f_1$ is constant, and further, $c=0$, i.e., $f_2$ is constant. In this case, $$\lambda=0, \ \ V^3=F,$$ where $F'=-\mu$. So, we get the conclusion.
\end{proof}

\begin{example}
For $f_1(x^2)=e^{x^2}$ and $f_2(x^1)=e^{x^1}$, the vector field
$$V=\frac{1}{2}e^{2x^3}\frac{\partial }{\partial x^3}$$ is the potential vector field of the almost $\eta$-Ricci soliton $$\left(\mathbb R^3, \tilde g, \lambda(x^1,x^2)=e^{2x^1}+e^{2x^2},\mu(x^1,x^2,x^3)=-(e^{2x^1}+e^{2x^2}+e^{2x^3})\right)$$ for $\eta=dx^3$.
\end{example}

\begin{theorem}\label{gsm}
Let $(I, \tilde g, V,\lambda, \mu)$ be an almost $\eta$-Ricci soliton with $\eta=dx^3$. If
\linebreak
$f_1=f_1(x^2)$, $f_2=f_2(x^3)$, and $V^i=V^i(x^3)$ for $i\in \{1,2,3\}$, then:
\begin{equation*}
\left\{
   \begin{aligned}
& V^1=c_1\\
& V^2=c_2f_2+\frac{c_3}{f_2}\\
& aV^1=0\\
& dV^3=aV^2+d'-d^2\\
& a=-c_2f_2+F\\
& \lambda=aV^2-f_2F'+a^2\\
& (V^3)'=-(d'-d^2)-(\lambda+\mu)\\
   \end{aligned}
  \right.,
\end{equation*}
where $c_1,c_2, c_3\in \mathbb R$, and $F$ is a smooth real function on $I_2$.
Moreover,

\quad (i) if $f_2$ is constant, then $V^2$ is constant;

\quad (ii) if $f_2$ is not constant, then $f_1(x^2)=c_4e^{-c_2x^2}$, where $c_4\in \mathbb R\setminus \{0\}$, and $a=-c_2f_2$;

\quad (iii) if $V^1\neq 0$, then $a=0$, $dV^3=d'-d^2$, $\lambda=-f_2F'$, and $f_1$ is constant.
If, in addition, $f_2$ is not constant, then $V^2=\displaystyle\frac{c_3}{f_2}$, $F=0$, $\lambda=0$, and $\mu=-(V^3)'-(d'-d^2)$.
\end{theorem}

\begin{proof}
In this case, we have $a=f_2\frac{\displaystyle f_1'}{\displaystyle f_1}$, $b=c=0$, $d=\frac{\displaystyle f_2'}{\displaystyle f_2}$, and
$$\R(E_1,E_1)=E_2(a)-a^2,\ \
\R(E_2,E_2)=E_2(a)-a^2+E_3(d)-d^2, \ \ \R(E_3,E_3)=E_3(d)-d^2,$$
$$\R(E_1,E_2)=\R(E_1,E_3)=0, \ \ \R(E_2,E_3)=E_3(a),$$
and (\ref{19}) becomes
\pagebreak
\begin{equation*}
\left\{
    \begin{aligned}
 &      f_1\frac{\displaystyle \partial V^1}{\displaystyle \partial x^1}-aV^2+f_2\frac{\displaystyle \partial a}{\displaystyle \partial x^2}-a^2+\lambda=0 \\
 &      f_2\frac{\displaystyle \partial V^2}{\displaystyle \partial x^2}-dV^3+d'-d^2+f_2\frac{\displaystyle \partial a}{\displaystyle \partial x^2}-a^2+\lambda=0\\
&      \frac{\displaystyle \partial V^3}{\displaystyle \partial x^3}+d'-d^2+\lambda+\mu=0\\
  &     f_1\frac{\displaystyle \partial V^2}{\displaystyle \partial x^1}+f_2\frac{\displaystyle \partial V^1}{\displaystyle \partial x^2}+aV^1=0 \\
  &     f_1\frac{\displaystyle \partial V^3}{\displaystyle \partial x^1}+\frac{\displaystyle \partial V^1}{\displaystyle \partial x^3}=0 \\
  &     \frac{\displaystyle 1}{\displaystyle 2}\left[f_2\frac{\displaystyle \partial V^3}{\displaystyle \partial x^2}+\frac{\displaystyle \partial V^2}{\displaystyle \partial x^3}+dV^2\right]+\frac{\displaystyle \partial a}{\displaystyle \partial x^3}=0
    \end{aligned}
  \right.,
\end{equation*}
which is equivalent to
\begin{equation}\label{22}
\left\{
    \begin{aligned}
 &      (V^1)'=0 \\
 &      aV^1=0\\
&      (V^2)'+dV^2=-2\frac{\displaystyle \partial a}{\displaystyle \partial x^3}\\
  &     aV^2=f_2\frac{\displaystyle \partial a}{\displaystyle \partial x^2}-a^2+\lambda \\
  &     (V^3)'=-(d'-d^2)-(\lambda+\mu) \\
  &     dV^3=f_2\frac{\displaystyle \partial a}{\displaystyle \partial x^2}-a^2+d'-d^2+\lambda
    \end{aligned}
  \right..
\end{equation}
From the first equation, it follows that $V^1$ is constant.
From the third equation, we deduce that $\displaystyle\frac{\partial a}{\partial x^3}$ depends only on $x^3$, while $a$ depends on $x^2$ and $x^3$.
On the other hand, we have $\displaystyle\frac{\partial a}{\partial x^3}=f_2'\displaystyle\frac{f_1'}{f_1}$.
Since $\displaystyle\frac{f_1'}{f_1}$ depends only on $x^2$, if there exists $x^3_0\in I_3$ such that $f_2'(x^3_0)\neq 0$, then $\displaystyle\frac{f_1'}{f_1}$ is constant on $I_2$, let's say $\displaystyle\frac{f_1'}{f_1}=c_1$ with $c_1\in \mathbb R$. So, in the case $f_2$ not constant, we infer that $f_1(x^2)=c_2e^{c_1x^2}$, where $c_2\in \mathbb R\setminus \{0\}$, and $a=c_1f_2$. Hence,
$\displaystyle\frac{\partial a}{\partial x^3}=c_1f_2'$, and we obtain the equation
$$(V^2)'+\displaystyle\frac{f_2'}{f_2}V^2=-2c_1f_2',$$
with the solution
$V^2=-c_1f_2+\displaystyle\frac{c_3}{f_2}$, where $c_3\in \mathbb R$.
If $f_2$ is constant on $I_3$, we obtain
$\displaystyle\frac{\partial a}{\partial x^3}=0$ and $d=0$, and the third equation of the system becomes $(V^2)'=0$, so $V^2=c_4$, where $c_4\in \mathbb R$.
We notice that we always have
$$\displaystyle\frac{\partial a}{\partial x^3}=c_1f_2'\ \ \textrm{and}\ \ V^2=-c_1f_2+\displaystyle\frac{c_3}{f_2},$$
where $c_1, c_3\in \mathbb R$. We infer that $a(x^2,x^3)=c_1f_2+F(x^2)$, where $F$ is a smooth real function on $I_2$. Hence, $\displaystyle\frac{\partial a}{\partial x^2}=F'$, and, from the fourth equation of the system, we get
$$\lambda=aV^2-f_2F'+a^2.$$
The fifth equation of (\ref{22}) gives
$$\mu=-(V^3)'-(d'-d^2)-\lambda,$$
while, from the last equation of the system and the expression of $\lambda$, we get
$$dV^3=aV^2+d'-d^2,$$
that is,
$$dV^3=-c_1f_2^2\displaystyle\frac{f_1'}{f_1}+c_3\displaystyle\frac{f_1'}{f_1}+d'-d^2.\qquad \qed$$
\renewcommand{\qed}{}
\end{proof}

\begin{example}
For $f_1(x^2)=e^{x^2}$ and $f_2(x^3)=e^{x^3}$, the vector field
$$V=(1-e^{2x^3})\frac{\partial }{\partial x^2}-e^{2x^3}\frac{\partial }{\partial x^3}$$
is a potential vector field of the almost $\eta$-Ricci soliton $$\left(\mathbb R^3, \tilde g, \lambda=1,\mu(x^3)=2e^{2x^3}\right)$$ for $\eta=dx^3$.
\end{example}

\subsection{Almost $\eta$-Ricci solitons with $\eta=dx^3$ and $V=\frac{\displaystyle \partial}{\displaystyle \partial x^3}$}

If $V^1=V^2=\eta^1=\eta^2=0$ and $V^3=\eta^3=1$, then the system (\ref{17}) becomes
\begin{equation}\label{19p}
\left\{
    \begin{aligned}
  &     -b+\R(E_1,E_1)+\lambda=0 \\
  &     -d+\R(E_2,E_2)+\lambda=0 \\
   &    \R(E_3,E_3)+\lambda+\mu=0 \\
   &    \R(E_1,E_3)=0 \\
   &    \R(E_2,E_3)=0
    \end{aligned}
  \right..
\end{equation}

\begin{theorem}\label{mwa}
Let $V=\frac{\displaystyle \partial }{\displaystyle \partial x^3}$ and $\eta=dx^3$. If $f_i=f_i(x^3)$ for $i\in \{1,2\}$, then $(I,\tilde g,V, \lambda,\mu)$ is an almost $\eta$-Ricci soliton if and only if
$$\lambda=\frac{\displaystyle f_1'}{\displaystyle f_1}\left(\frac{\displaystyle f_2'}{\displaystyle f_2}+1\right)-\frac{\displaystyle f_1''}{\displaystyle f_1}+2\left(\frac{\displaystyle f_1'}{\displaystyle f_1}\right)^2\ \ \textrm{and} \ \
\mu=-\frac{\displaystyle f_1'}{\displaystyle f_1}\left(\frac{\displaystyle f_2'}{\displaystyle f_2}+1\right)-\frac{\displaystyle f_2''}{\displaystyle f_2}+2\left(\frac{\displaystyle f_2'}{\displaystyle f_2}\right)^2,$$
and $f_1$ and $f_2$ satisfy
\begin{align}\label{ka}
\frac{\displaystyle f_1''-f_1'}{\displaystyle f_1}-2\left(\frac{\displaystyle f_1'}{\displaystyle f_1}\right)^2&=\frac{\displaystyle f_2''-f_2'}{\displaystyle f_2}-2\left(\frac{\displaystyle f_2'}{\displaystyle f_2}\right)^2.
\end{align}
\end{theorem}
\begin{proof}
In this case, we have $a=0$, $b=\frac{\displaystyle f_1'}{\displaystyle f_1}$, $c=0$, $d=\frac{\displaystyle f_2'}{\displaystyle f_2}$, and (\ref{19p}) becomes
\begin{equation}\label{17aax}
\left\{
    \begin{aligned}
 &      -b+\R(E_1,E_1)+\lambda=0 \\
 &      -d+\R(E_2,E_2)+\lambda=0 \\
 &      \R(E_3,E_3)+\lambda+\mu=0
    \end{aligned}
  \right.,
\end{equation}
with
\begin{align*}
\R(E_1,E_1)&=\frac{f_1''}{f_1}-2\left(\frac{f_1'}{f_1}\right)^2-\frac{f_1'}{f_1}\frac{f_2'}{f_2},\\
\R(E_2,E_2)&=\frac{f_2''}{f_2}-2\left(\frac{f_2'}{f_2}\right)^2-\frac{f_1'}{f_1}\frac{f_2'}{f_2},\\
\R(E_3,E_3)&=\frac{f_1''}{f_1}-2\left(\frac{f_1'}{f_1}\right)^2+\frac{f_2''}{f_2}-2\left(\frac{f_2'}{f_2}\right)^2.
\end{align*}
The two functions $f_1$ and $f_2$ must satisfy (\ref{ka}) due to the first two equations of (\ref{17aax}), and we obtain the expressions of $\lambda$ and $\mu$ from the first and the last equation of the same system.
\end{proof}

\begin{example}
For $f_i(x^3)=k_ie^{kx^3}$, $i\in \{1,2\}$, $k,k_1,k_2\in \mathbb R\setminus \{0\}$,
$V=\frac{\displaystyle \partial }{\displaystyle \partial x^3}$, and $\eta=dx^3$, $(\mathbb R^3,\tilde g, V, \lambda, \mu)$ is an $\eta$-Ricci soliton if and only if $$\lambda=2k^2+k\ \ \textrm{and} \ \ \mu=-k.$$
\end{example}

\begin{corollary}
Under the hypotheses of Theorem \ref{mwa}, if $f_1=f_2=:f(x^3)$, then
$(I,\tilde g,V)$ is an almost $\eta$-Ricci soliton with $\lambda$ and $\mu$ as scalar functions if and only if they are given by
$$\lambda=3\left(\frac{\displaystyle f'}{\displaystyle f}\right)^2+\frac{\displaystyle f'-f''}{\displaystyle f}\ \ \textrm{and} \ \
\mu=\left(\frac{\displaystyle f'}{\displaystyle f}\right)^2-\frac{\displaystyle f'+f''}{\displaystyle f}.$$
\end{corollary}
\begin{proof}
It follows immediately from Theorem \ref{mwa}.
\end{proof}
\pagebreak

\begin{theorem}\label{h}
Let $V=\frac{\displaystyle \partial }{\displaystyle \partial x^3}$ and $\eta=dx^3$. If $f_1=f_1(x^1)$, $f_2=f_2(x^3)$,
then $(I,\tilde g,V,\lambda,\mu)$ is an almost $\eta$-Ricci soliton if and only if
$$f_2(x^3)=c_0e^{-x^3}, \ \ \lambda=0, \ \ \mu=1,$$
where $c_0\in \mathbb R\setminus \{0\}$, or
$$f_2(x^3)=\frac{c_1}{c_2e^{x^3}+1}, \ \ \lambda=0, \ \ \mu(x^3)= \frac{c_2e^{x^3}}{c_2e^{x^3}+1},$$
where $c_1,c_2\in \mathbb R\setminus \{0\}$ such that $c_2\neq -e^{-x^3}$ for all $x^3\in I_3$.

In particular, for $I_3=\mathbb R$, we additionally have the condition $c_2>0$.
\end{theorem}
\begin{proof}
In this case, we have $a=b=c=0$, $d=\frac{\displaystyle f_2'}{\displaystyle f_2}$, and (\ref{19p}) becomes
\begin{equation*}
\left\{
    \begin{aligned}
      &\lambda=0 \\
      &-d+d'-d^2=0 \\
      &d'-d^2+\mu=0
    \end{aligned}
  \right..
\end{equation*}
If $d=0$, then $f_2$ is constant, and $\lambda=\mu=0$ (which contradicts $\mu\neq 0$). If $d=-1$, then $\ln |f_2(x^3)|=-x^3+k_0$, where $k_0\in \mathbb R$, and we get $f_2(x^3)=\pm e^{-x^3+k_0}$ and $\mu=1$. If $d\neq -1$ and $d\neq 0$, from the second equation of the above system, we get
$$1=\frac{d'}{d^2+d}=\frac{d'}{d}-\frac{d'}{d+1},$$
which, by integration, gives
$$\ln\left(\left|\frac{d(x^3)}{d(x^3)+1}\right|\right)=x^3+k_1,$$
where $k_1\in \mathbb R$, from which
$$(1-k_2e^{x^3})d(x^3)=k_2e^{x^3},$$
where $k_2\in \mathbb R\setminus \{0\}$ and for all $x^3\in I_3$.
This infers
$$\frac{f_2'(x^3)}{f_2(x^3)}=d(x^3)=\frac{k_2e^{x^3}}{1-k_2e^{x^3}},$$
where $k_2\neq e^{-x^3}$ for all $x^3\in I_3$, and, by integrating, we get
$$f_2(x^3)=\frac{k_3}{1-k_2e^{x^3}},$$
where $k_3\in \mathbb R\setminus \{0\}$, and
$$\mu(x^3)=d^2(x^3)-d'(x^3)=-d(x^3)=\frac{-k_2e^{x^3}}{1-k_2e^{x^3}}.\qquad \qed$$
\renewcommand{\qed}{}
\end{proof}

\begin{theorem}\label{mw1}
Let $V=\frac{\displaystyle \partial }{\displaystyle \partial x^3}$ and $\eta=dx^3$. If $f_i=f_i(x^2)$ for $i\in \{1,2\}$, then $(I,\tilde g,V, \lambda, \mu)$ is an almost $\eta$-Ricci soliton if and only if
$$\lambda=-f_2^2\left[\left(\frac{\displaystyle f_1'}{\displaystyle f_1}\right)'-\left(\frac{\displaystyle f_1'}{\displaystyle f_1}\right)^2+\frac{f_1'}{f_1}\frac{f_2'}{f_2}\right]\ \ \textrm{and} \ \
\mu=f_2^2\left[\left(\frac{\displaystyle f_1'}{\displaystyle f_1}\right)'-\left(\frac{\displaystyle f_1'}{\displaystyle f_1}\right)^2+\frac{f_1'}{f_1}\frac{f_2'}{f_2}\right].$$
\end{theorem}
\begin{proof}
In this case, we have $a=f_2\frac{\displaystyle f_1'}{\displaystyle f_1}$, $b=c=d=0$, and (\ref{19p}) becomes
\begin{equation*}
\left\{
    \begin{aligned}
 &      \R(E_1,E_1)+\lambda=0 \\
  &     \R(E_2,E_2)+\lambda=0 \\
  &     \lambda+\mu=0
    \end{aligned}
  \right.,
\end{equation*}
with
$$\R(E_1,E_1)=\R(E_2,E_2)=f_2^2\left[\left(\frac{\displaystyle f_1'}{\displaystyle f_1}\right)'-\left(\frac{\displaystyle f_1'}{\displaystyle f_1}\right)^2+\frac{f_1'}{f_1}\frac{f_2'}{f_2}\right]. \qquad \qed$$
\renewcommand{\qed}{}
\end{proof}

\begin{example}
For $f_1(x^2)=k_1e^{-x^2}$, $f_2(x^2)=k_2e^{x^2}$, $k_1,k_2\in \mathbb R\setminus \{0\}$, $V=\frac{\displaystyle \partial }{\displaystyle \partial x^3}$, and $\eta=dx^3$, $( \mathbb R^3,\tilde g, V, \lambda, \mu)$ is an almost $\eta$-Ricci soliton if and only if $$\lambda(x^2)=2k_2^2e^{2x^2}\ \ \textrm{and} \ \ \mu(x^2)=-2k_2^2e^{2x^2}.$$
\end{example}

\begin{corollary}
Under the hypotheses of Theorem \ref{mw1}, if $f_1=f_2=:f(x^2)$, then
$(I,\tilde g,V)$ is an almost $\eta$-Ricci soliton with $\lambda$ and $\mu$ as scalar functions if and only if they are given by
$$\lambda=(f')^2-f''f\ \ \textrm{and} \ \
\mu=-(f')^2+f''f.$$
\end{corollary}
\begin{proof}
It follows immediately from Theorem \ref{mw1}.
\end{proof}

\begin{theorem}\label{mw2}
Let $V=\frac{\displaystyle \partial }{\displaystyle \partial x^3}$ and $\eta=dx^3$. If $f_1=f_1(x^2)$ and $f_2=f_2(x^1)$, then $(I,\tilde g,V, \lambda,\mu)$ is an almost $\eta$-Ricci soliton if and only if
$$\lambda=-f_1^2\left[\left(\frac{\displaystyle f_2'}{\displaystyle f_2}\right)'-\left(\frac{\displaystyle f_2'}{\displaystyle f_2}\right)^2\right]-f_2^2\left[\left(\frac{\displaystyle f_1'}{\displaystyle f_1}\right)'-\left(\frac{\displaystyle f_1'}{\displaystyle f_1}\right)^2\right]$$
and
$$\mu=f_1^2\left[\left(\frac{\displaystyle f_2'}{\displaystyle f_2}\right)'-\left(\frac{\displaystyle f_2'}{\displaystyle f_2}\right)^2\right]+f_2^2\left[\left(\frac{\displaystyle f_1'}{\displaystyle f_1}\right)'-\left(\frac{\displaystyle f_1'}{\displaystyle f_1}\right)^2\right].$$
\end{theorem}
\begin{proof}
In this case, we have $a=f_2\frac{\displaystyle f_1'}{\displaystyle f_1}$, $b=0$, $c=f_1\frac{\displaystyle f_2'}{\displaystyle f_2}$, $d=0$, and (\ref{19p}) becomes
\begin{equation*}
\left\{
    \begin{aligned}
  &     \R(E_1,E_1)+\lambda=0 \\
  &     \R(E_2,E_2)+\lambda=0 \\
   &    \lambda+\mu=0
    \end{aligned}
  \right.,
\end{equation*}
with
$$\R(E_1,E_1)=\R(E_2,E_2)=f_1^2\left[\left(\frac{\displaystyle f_2'}{\displaystyle f_2}\right)'-\left(\frac{\displaystyle f_2'}{\displaystyle f_2}\right)^2\right]+f_2^2\left[\left(\frac{\displaystyle f_1'}{\displaystyle f_1}\right)'-\left(\frac{\displaystyle f_1'}{\displaystyle f_1}\right)^2\right]. \qquad \qed$$
\renewcommand{\qed}{}
\end{proof}

\begin{example}
For $f_1(x^2)=e^{x^2}$, $f_2(x^1)=e^{x^1}$, $V=\frac{\displaystyle \partial }{\displaystyle \partial x^3}$, and $\eta=dx^3$, $(\mathbb R^3,\tilde g, V, \lambda, \mu)$ is an almost $\eta$-Ricci soliton if and only if $$\lambda(x^1,x^2)=e^{2x^1}+e^{2x^2}\ \ \textrm{and} \ \ \mu(x^1,x^2)=-e^{2x^1}-e^{2x^2}.$$
\end{example}

\begin{theorem}\label{mw}
Let $V=\frac{\displaystyle \partial }{\displaystyle \partial x^3}$ and $\eta=dx^3$. If $f_1=f_1(x^2)$ and $f_2=f_2(x^3)$, then $(I,\tilde g,V, \lambda,\mu)$ is an almost $\eta$-Ricci soliton if and only if
$$d'=d(d+1), \ \  f_1' f_2'=0,$$
$$\lambda=-f_2^2\left[\left(\frac{\displaystyle f_1'}{\displaystyle f_1}\right)'-\left(\frac{\displaystyle f_1'}{\displaystyle f_1}\right)^2\right]\ \ \textrm{and} \ \  \mu=f_2^2\left[\left(\frac{\displaystyle f_1'}{\displaystyle f_1}\right)'-\left(\frac{\displaystyle f_1'}{\displaystyle f_1}\right)^2\right]-\left[\left(\frac{\displaystyle f_2'}{\displaystyle f_2}\right)'-\left(\frac{\displaystyle f_2'}{\displaystyle f_2}\right)^2\right].$$

Moreover, one of the following cases is satisfied:

(i) $f_2=k_2\in \mathbb R\setminus\{0\}$ and
$$\lambda=-f_2^2\left[\left(\frac{\displaystyle f_1'}{\displaystyle f_1}\right)'-\left(\frac{\displaystyle f_1'}{\displaystyle f_1}\right)^2\right],\ \
\mu=f_2^2\left[\left(\frac{\displaystyle f_1'}{\displaystyle f_1}\right)'-\left(\frac{\displaystyle f_1'}{\displaystyle f_1}\right)^2\right];$$

(ii) $f_1=k_1\in \mathbb R\setminus\{0\}$, $f_2(x^3)=\frac{\displaystyle c_1}{\displaystyle c_2e^{x^3}+1}$, where $c_1,c_2\in \mathbb R\setminus \{0\}$ such that $c_2\neq -e^{-x^3}$ for all $x^3\in I_3$, and
$$\lambda =0,\ \ \mu=\frac{\displaystyle c_2e^{x^3}}{\displaystyle \displaystyle c_2e^{x^3}+1};$$

(iii) $f_1=k_1\in \mathbb R\setminus\{0\}$, $f_2(x^3)=c_0 e^{-x^3}$, where $c_0\in \mathbb R\setminus \{0\}$, and
$$\lambda =0,\ \ \mu=1.$$
\end{theorem}
\begin{proof}
In this case, we have $a=f_2\frac{\displaystyle f_1'}{\displaystyle f_1}$, $b=c=0$, $d=\frac{\displaystyle f_2'}{\displaystyle f_2}$, and (\ref{19p}) becomes
\begin{equation}\label{17aa}
\left\{
    \begin{aligned}
 &      \R(E_1,E_1)+\lambda=0 \\
 &      -d+\R(E_2,E_2)+\lambda=0 \\
 &      \R(E_3,E_3)+\lambda+\mu=0 \\
 &      \R(E_2,E_3)=0
    \end{aligned}
  \right.,
\end{equation}
with
\begin{align*}
\R(E_1,E_1)&=f_2^2\left[\left(\frac{f_1'}{f_1}\right)'-\left(\frac{f_1'}{f_1}\right)^2\right],\\
\R(E_2,E_2)&=f_2^2\left[\left(\frac{f_1'}{f_1}\right)'-\left(\frac{f_1'}{f_1}\right)^2\right]+\left(\frac{f_2'}{f_2}\right)'-\left(\frac{f_2'}{f_2}\right)^2,\\
\R(E_3,E_3)&=\left(\frac{f_2'}{f_2}\right)'-\left(\frac{f_2'}{f_2}\right)^2,\\
\R(E_2,E_3)&=\frac{f_1'f_2'}{f_1}.
\end{align*}
$f_2$ must satisfy $d'=d(d+1)$ due to the first two equations of (\ref{17aa}), and we obtain the cases: \\
(i) $f_2$ constant; \\
(ii) $f_1$ constant and $f_2(x^3)=\frac{\displaystyle c_1}{\displaystyle c_2e^{x^3}+1}$ on $I_3$, where $c_1,c_2\in \mathbb R\setminus \{0\}$ and $e^{x^3}\neq -\frac{\displaystyle 1}{\displaystyle c_2}$ on $I_3$; \\
(iii) $f_1$ constant and $f_2(x^3)=c_0 e^{-x^3}$, where $c_0\in \mathbb R\setminus \{0\}$.
\end{proof}

\begin{example}
For $f_1=1$, $f_2(x^3)=e^{-x^3}$,
$V=\frac{\displaystyle \partial }{\displaystyle \partial x^3}$, and $\eta=dx^3$, $(\mathbb R^3,\tilde g, V, \lambda, \mu)$ is an $\eta$-Ricci soliton if and only if $$\lambda=0\ \ \textrm{and} \ \ \mu=1.$$ 
\end{example}


\vspace{0.001cm}

\textit{Department of Mathematics}

\textit{West University of Timi\c{s}oara}

\textit{Faculty of Mathematics and Computer Science}

\textit{Bld. V. P\^{a}rvan no. 4, 300223, Timi\c{s}oara, Rom\^{a}nia}

\textit{adarablaga@yahoo.com}


\begin{thebibliography}{33}

\bibitem{ata} P. Atashpeykara, A. Haji-Badali, \textit{The algebraic Ricci solitons of Lie groups $\mathbb H^2\times \mathbb R$ and $Sol_3$}, {Journal of Finsler Geometry and its Applications}
{\bf 1(2)} (2020), 105--114.
https://doi.org/10.22098/jfga.2020.1244

\bibitem{bel2} L. Belarbi, \textit{Ricci solitons of the $\mathbb H^2\times \mathbb R$ Lie group}, {Electronic Research Archive} {\bf 28(1)} (2020), 157--163.\\
https://doi.org/10.3934/era.2020010

\bibitem{bel1} L. Belarbi, \textit{Ricci solitons of the $Sol_3$ Lie group}, {Authorea.} July 26 (2020).
https://doi.org/10.22541/au.159576131.13724173

\bibitem{b} A.M. Blaga, \textit{Almost $\eta$-Ricci solitons in $(LCS)_n$-manifolds}, {Bull. Belg. Math. Soc. Simon Stevin} {\bf 25(5)} (2018), 641--653.
https://doi.org/10.36045/bbms/1547780426

\bibitem{adara} A.M. Blaga, \textit{$\eta$-Ricci solitons on para-Kenmotsu manifolds}, {Balkan J. Geom. Appl.} {\bf 20(1)} (2015), 1--13.

\bibitem{bl} A.M. Blaga, \textit{Solutions of some types of soliton equations in $\mathbb{R}^3$}, {Filomat} {\bf 33(4)} (2019), 1159--1162.
https://doi.org/10.2298/FIL1904159B

\bibitem{cao3} H.-D. Cao, X. Cui, \textit{Curvature estimates for four-dimensional gradient steady Ricci solitons}, {J. Geom. Anal.} {\bf 30(1)} (2020), 511--525.
https://doi.org/10.1007/s12220-019-00152-z

\bibitem{cao2} H.-D. Cao, T. Liu, \textit{Curvature estimates for $4$-dimensional complete gradient expanding Ricci solitons}, {Journal f\"ur die reine und angewandte Mathematik} {\bf 790} (2022), 115--135.
https://doi.org/10.1515/crelle-2022-0039

\bibitem{cao1} H.-D. Cao, M. Zhu, \textit{Linear stability of compact shrinking Ricci solitons}, {Math. Ann.} (2024), 1--17.
https://doi.org/10.1007/s00208-024-02824-w

\bibitem{kimura} J.T. Cho, M. Kimura, \textit{Ricci solitons and real hypersurfaces in a complex space form}, {Tohoku Math. J.} {\bf 61(2)} (2019), 205--212.
https://doi.org/10.2748/tmj/1245849443

\bibitem{chlu} B. Chow, P. Lu, L. Ni, \textit{Hamilton's Ricci Flow}, Graduate Studies in Mathematics 77, AMS, Providence, RI, USA (2006).

\bibitem{ham} R.S. Hamilton, \textit{The Ricci flow on surfaces}, Math. and general relativity (Santa Cruz, CA, 1986), {Contemp. Math.} {\bf 71} (1988), AMS, 237--262.
https://doi.org/10.1090/conm/071/954419

\bibitem{pi} S. Pigola, M. Rigoli, M. Rimoldi, A.G. Setti, \textit{Ricci almost solitons}, {Ann. Scuola Norm. Sup. Pisa Cl. Sci.} {\bf 10(4)} (2011), 757--799.

\end{thebibliography}
\end{document}